\documentclass[english]{article}
\usepackage{amsmath}
\usepackage{amsthm}
\usepackage{amssymb}
\usepackage{authblk}
\usepackage{color}
\usepackage{babel}
\usepackage{mathrsfs}
\usepackage{graphicx}
\usepackage[bookmarks=true,bookmarksnumbered=false,bookmarksopen=false,
 breaklinks=false,pdfborder={0 0 1},backref=false,colorlinks=true]
 {hyperref}

\makeatletter
\ifx\proof\undefined
\newenvironment{proof}[1][\protect\proofname]{\par
	\normalfont\topsep6\p@\@plus6\p@\relax
	\trivlist
	\itemindent\parindent
	\item[\hskip\labelsep\scshape #1]\ignorespaces
}{%
	\endtrivlist\@endpefalse
}
\providecommand{\proofname}{Proof}
\fi

\@ifundefined{date}{}{\date{}}
\ifdefined\showcaptionsetup
 \PassOptionsToPackage{caption=false}{subfig}
\fi
\usepackage{subfig}
\makeatother

\theoremstyle{plain}
\newtheorem{thm}{\protect\theoremname}[section]
\newtheorem{cor}[thm]{\protect\corollaryname}
\theoremstyle{remark}
\newtheorem{rem}[thm]{\protect\remarkname}
\theoremstyle{plain}
\newtheorem{prop}[thm]{\protect\propositionname}
\newtheorem{lem}[thm]{\protect\lemmaname}
\theoremstyle{remark}

\providecommand{\claimname}{Claim}
\providecommand{\corollaryname}{Corollary}
\providecommand{\lemmaname}{Lemma}
\providecommand{\propositionname}{Proposition}
\providecommand{\remarkname}{Remark}
\providecommand{\theoremname}{Theorem}

\title{Long-time stability of hierarchical point vortex configurations}

\author[a]{Slim Ibrahim}
\author[b]{Shengyi Shen}

\affil[a]{Department of Mathematics and Statistics, University of Victoria, Canada}
\affil[b]{School of Mathematics and Statistics, Ningbo University, China}

\date{}

\begin{document}
\maketitle

\begin{abstract}
In this paper, we investigate the dynamics and long-time stability of hierarchical configurations in an $N$-point vortex system on $\mathbb{R}^2$ and in simply connected domains, without the Cantor-set restriction inherent in KAM theory. Since the Hamiltonian perturbation lacks a standard power-series expansion in the usual normalized variables, we develop a modified Birkhoff normal form based on successive ratios of adjacent variables, tailored to the hierarchical structure. A key feature is the preservation of the resulting natural exponent structure throughout the normal form procedure. Consequently, for $n$ sufficiently large, the hierarchical configuration persists for times of order at least $\varepsilon^{-\left(n-2(N-1)\right)}$.
\end{abstract}

\section{Introduction}

The point vortex system is one of the most classical models and influential finite dimensional reductions in fluid mechanics. Beyond their intrinsic relevance to hydrodynamics, point vortex systems have also become fundamental examples in Hamiltonian dynamics, geometric mechanics, and singular perturbation theory. 
A central theme in vortex dynamics is the existence and stability of coherent structures, particularly relative equilibria, namely vortex configurations evolving by rigid uniform rotation or translation~\cite{aref2003vortex,newton2001n}. Although classical work focused on symmetric configurations such as Thomson polygons~\cite{thomson1883treatise}, recent studies have revealed rich geometric families and bifurcation structures of asymmetric relative equilibria, including a detailed classification for the planar four-vortex problem~\cite{kallyadan2025geometries}. Finiteness of stationary configurations has been established for the planar four-vortex problem~\cite{yu2023finiteness}. For the planar five-vortex problem, the corresponding result holds for generic vortex strengths~\cite{yu2025finiteness}. Beyond rigid motions, self-similar collapse and expansion configurations have also been studied through a linear-algebraic formulation of the configuration equations~\cite{kallyadan2022selfsimilar}. Such expanding point-vortex configurations have further been used to construct global solutions of the two-dimensional Euler equation consisting of well-separated vortex patches~\cite{zbarsky2021point}.

A prominent focus in recent years is the study of hierarchical or multiscale vortex structures, where vortices form nested clusters interacting at different spatial scales~\cite{wayne2012}. These configurations present delicate dynamical features due to resonant interactions between internal cluster motions and global collective dynamics. Understanding their long-term persistence requires tools capable of capturing these multiscale resonances.

This paper investigates the long-time stability of broad classes of hierarchical configurations for the $N$-point vortex problem in both the plane and bounded simply connected domains. The core novelty is the systematic application of Jacobi-type hierarchical coordinates combined with high-order adapted Birkhoff normal form techniques~\cite{birkhoff1927dynamical, giorgilli2002notes}. 

Although Birkhoff normal forms are standard for near-integrable Hamiltonian systems, their application to vortex dynamics poses unique challenges: the logarithmic interaction potential creates strong geometric degeneracies, nested clusters produce complicated frequency resonances, and boundary effects introduce nontrivial couplings. We overcome these obstacles by employing hierarchical Jacobi coordinates adapted~\cite{lim1998jacobi}, which naturally separate the center-of-mass motion from internal cluster oscillations, partially diagonalizing the leading-order interactions. 

A natural approach to study the stability consists of expanding the dynamics about a prescribed reference trajectory. Since this generally produces a nonautonomous Hamiltonian system, with periodic or quasiperiodic couplings persisting even in a co-rotating frame, the classical autonomous Birkhoff normal form near a fixed equilibrium is not directly applicable. Instead, after suitable symplectic reductions, we retain the autonomous dominant Hamiltonian, exploit successive ratios of adjacent Jacobi variables, and solve a modified homological equation to construct high-order normal forms in the hierarchical regime while controlling resonant interactions across cluster scales. This yields stability for a full neighborhood of hierarchical motions and shows that configurations initially close to a hierarchically stable relative equilibrium remain close for polynomially long times, with the exponent determined by the normalization order.



At this point, it is natural to compare our framework with KAM approaches to hierarchical vortex systems. The existence of perpetual quasi-periodic motions in such configurations was first established by Khanin~\cite{khanin1982}, and more recently, Xiong and Liu~\cite{liu2023} derived explicit nondegeneracy conditions for KAM tori in planar $N$-point vortex systems. KAM theory provides invariant quasi-periodic tori for all time for a Cantor family of initial data satisfying nonresonance conditions. By contrast, our main theorem establishes quantitative stability over polynomially long times for all initial data satisfying the hierarchical assumptions stated below, without imposing a Diophantine condition. The underlying normal form scheme is not restricted to this particular hierarchy and can be adapted to broader multiscale configurations, as illustrated by the examples in Section~\ref{sec:app}. In the infinite-dimensional Euler setting, KAM methods have also produced time-quasi-periodic vortex patches near uniformly rotating Kirchhoff ellipses~\cite{berti2023time}; see also~\cite{berti2024kam} for a survey.

Our methodology draws inspiration from celestial mechanics, where Jacobi coordinates and normal forms are central to the stability theory of hierarchical $N$-body systems~\cite{meyer1992introduction}. However, the noncanonical symplectic structure and logarithmic potential of the vortex problem generate distinct resonance phenomena that require specialized treatment. 

The paper is organized as follows: Section 2 reviews the Hamiltonian formulation and introduces some preliminary work. Section 3 develops the formal normal form procedure. The rigorous Birkhoff theorem is proved in Section 4, followed by stability estimates (Section 5). Examples of nested configurations conclude the paper (Section 6).

\subsection{The setup and main results}
In the classical Birkhoff Normal Form theory, the unperturbed Hamiltonian
$h_{0}$ is a quadratic polynomial, making the homological operator
$\left\{ \cdot,h_{0}\right\} $ strictly degree-preserving. Consequently,
the homological equation can be solved independently within each subspace
of homogeneous polynomials. However, the presence of logarithmic singularities
in our $h_{0}$ induces a non-homogeneous behavior. The Poisson bracket
$\left\{ \cdot,h_{0}\right\} $ breaks the classical algebraic bahaviours
and coupling terms of different orders. Another algebra difficulty
is that the perturbation is not even a power series in normalized
variable $\left(\boldsymbol{\alpha},\bar{\boldsymbol{\alpha}}\right)$.
We introduce a new variable $\left(\boldsymbol{\beta},\bar{\boldsymbol{\beta}}\right)$
based on the successive ratios of $\left(\boldsymbol{\alpha},\bar{\boldsymbol{\alpha}}\right)$
that enable us to rewrite the perturbation as power series. Although the Poisson
bracket $\left\{ \cdot,h_{0}\right\} _{\mathcal{S}}$ still output
inhomogeneous polynomials in $\left(\boldsymbol{\beta},\bar{\boldsymbol{\beta}}\right)$,
one can solve the homological equation by just adding a higher order
term. 

Let $A_{k}=\sum^{k}_{l=1}a_{l}$ where $a_{l}$ are vortex strength.
The classical Jacobi coordinates (see Fig \ref{fig:jacobi_coordinates})
are defined by
\begin{align*}
w_{k} & =\frac{\sum^{k}_{l=1}a_{l}z_{l}}{A_{k}}-z_{k+1},k=1,2,\cdots,N-1,\\
w_{N} & =\frac{\sum^{N}_{l=1}a_{l}z_{l}}{A_{N}}.
\end{align*}
A key feature of Jacobi coordinates is their ability to effectively
describe relative particle motions. 
\begin{figure}
\begin{centering}
\includegraphics[scale=0.4]{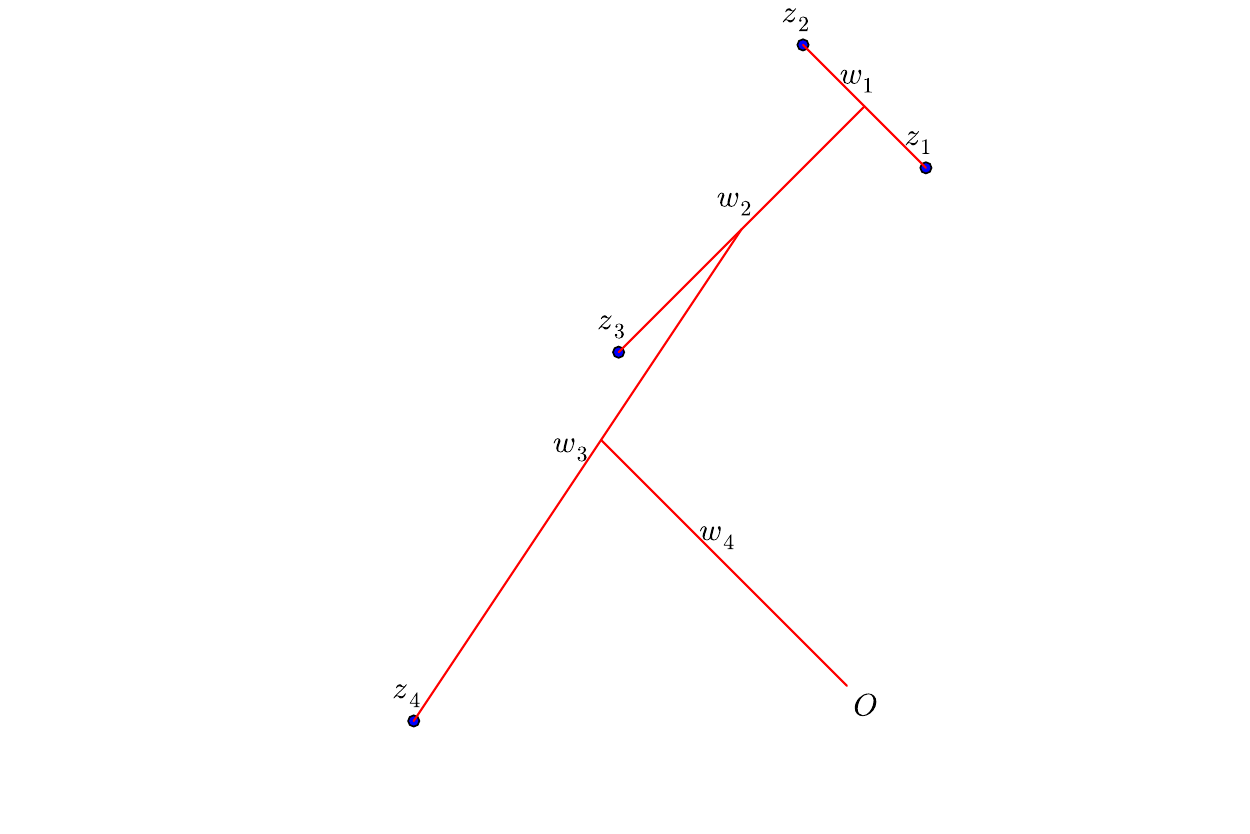}
\par\end{centering}
\caption{Jacobi coordinates for $N=4$.\label{fig:jacobi_coordinates}}

\end{figure}

Qualitatively, a hierarchical configuration consists of nested vortex
clusters on well-separated spatial scales: for each $k=2,\ldots,N-1$,
the vortices $1,\ldots,k$ form a cluster whose diameter is much smaller
than its distance from the $(k+1)$-st vortex. In Jacobi coordinates,
this scale separation is expressed by
$|w_1(0)|\ll\cdots\ll|w_{N-1}(0)|$. The precise
$\varepsilon$-dependent assumptions are stated in
Theorem~\ref{thm:main_result}, which shows that this hierarchy persists
for polynomially long times.
\begin{thm}
\label{thm:main_result}Given a simply connected domain $\Omega\subset\mathbb{R}^{2}$
with $0\in\Omega$ as the stationary point. Let $\varphi:\Omega\to\mathbb{D}$
be biholomorphic such that $\varphi\left(0\right)=0$. Assume that
$A_{k}\neq0,k=1,2,\cdots,N$ and the stability condition is satisfied
\[
2\left|\varphi^{\prime}\left(0\right)\right|^{3}>\left|\varphi^{\prime\prime\prime}\left(0\right)\right|.
\]
Then given $n>2\left(N-1\right)$, there exists $\varepsilon_{0}$ small
enough so that if the initial positions satisfy the following condition:
given $\varepsilon\le\varepsilon_{0}$
\begin{itemize}
\item $\frac{\left|w_{l}\left(0\right)\right|}{\left|w_{l+1}\left(0\right)\right|}=O\left(\varepsilon\right),l=1,2,\cdots,N-2$,
\item $\left|w_{N-1}\left(0\right)\right|=\varepsilon$,
\item $\left|w_{N}\left(0\right)\right|\le\varepsilon$,
\end{itemize}
there is a constant $C_{n}$ independent on $\varepsilon$ such that
\begin{align*}
\left|\left|w_{l}\left(t\right)\right|^{2}-\left|w_{l}\left(0\right)\right|^{2}\right| & \le C_{n}\varepsilon^{2\left(N-l\right)+1},1\le l\le N-1,\\
\left|\left|w_{N}\left(t\right)\right|^{2}-\left|w_{N}\left(0\right)\right|^{2}\right| & \le C_{n}\varepsilon^{3}.
\end{align*}
for $t\le\frac{1}{\varepsilon^{n-2\left(N-1\right)}}$.
\end{thm}
For the problem in the whole plane $\mathbb{R}^{2}$, the center of
vorticity $w_N$ is conserved and may be fixed at the origin by a
translation. This normalization concerns only the center of vorticity
and does not require the individual vortices to lie near the origin.
Thus, unlike Theorem~\ref{thm:main_result}, the whole-plane result
requires no $\varepsilon$-dependent smallness assumption on the outer
scale $|w_{N-1}(0)|$; only the hierarchy between successive Jacobi
scales is required.
Noticing that the Hamiltonian now simplifies to
\begin{align*}
H_{\mathbb{R}^{2}} & =\underbrace{-\sum^{N-1}_{k=1}\frac{a_{k+1}A_{k}}{4\pi}\ln\left(2w_{k}\bar{w}_{k}\right)}_{h_{0}}\\
 & \quad-\sum_{k=1}\sum_{l<k}\frac{a_{k}a_{l}}{2\pi}\ln\left|1+\left(\kappa_{l}-1\right)\frac{w_{l-1}}{w_{k-1}}+\sum^{k-2}_{j=l}\kappa_{j+1}\frac{w_{j}}{w_{k-1}}\right|,
\end{align*}
where $h_{0}$ is considered as its unperturbed part and the perturbation
only depends on the ratio of two adjacent components of $\boldsymbol{w}$.
Consequently, setting
$w_N\equiv0$ in the normal-form argument yields the following result.
\begin{cor}
Consider the $N$-point vortices system in $\mathbb{R}^{2}$ with corresponding $A_{k}\neq0,k=1,2,\cdots,N$. Given $n>2\left(N-1\right)$, there exists
$\varepsilon_{0}$ small enough so that if the initial positions satisfy
the following condition: given $\varepsilon\le\varepsilon_{0}$
\begin{itemize}
\item $\frac{\left|w_{l}\left(0\right)\right|}{\left|w_{l+1}\left(0\right)\right|}=O\left(\varepsilon\right),l=1,2,\cdots,N-2$,
\end{itemize}
there is a constant $C_{n}$ independent on $\varepsilon$ such that for $t\le\frac{1}{\varepsilon^{n-2\left(N-1\right)}}$
\[
w_{N}\left(t\right)\equiv0,\left|\left|w_{l}\left(t\right)\right|^{2}-\left|w_{l}\left(0\right)\right|^{2}\right|\le C_{n}\varepsilon^{2\left(N-l\right)+1},l=1,2,\cdots,N-1.
\]
\end{cor}
\begin{rem} When $N=3$, the dynamics resembles that of two point vortices. Indeed, $z_{1},z_{2}$
will rotate around their local center of vorticity. Simultaneously,
this local center acts as an effective point vortex, interacting with
$z_{3}$ to rotate around the global center of vorticity. The system
maintains this hierarchical motion for a long time. See Fig \ref{fig:3pt_example}. Our result can be seen as a generalization of this basic model.
\begin{figure}
\begin{raggedright}
\subfloat[\label{fig:3pt_jacobi_0}]{\includegraphics[scale=0.4]{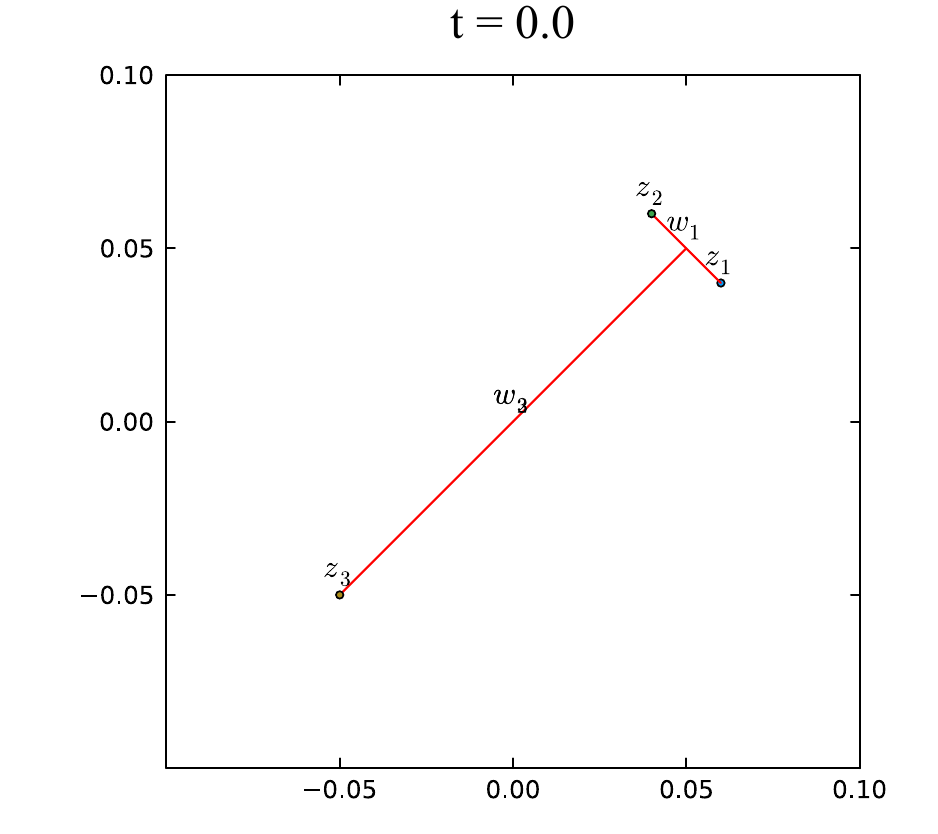}

}\subfloat[\label{fig:3pt_jacobi_1}]{\includegraphics[scale=0.4]{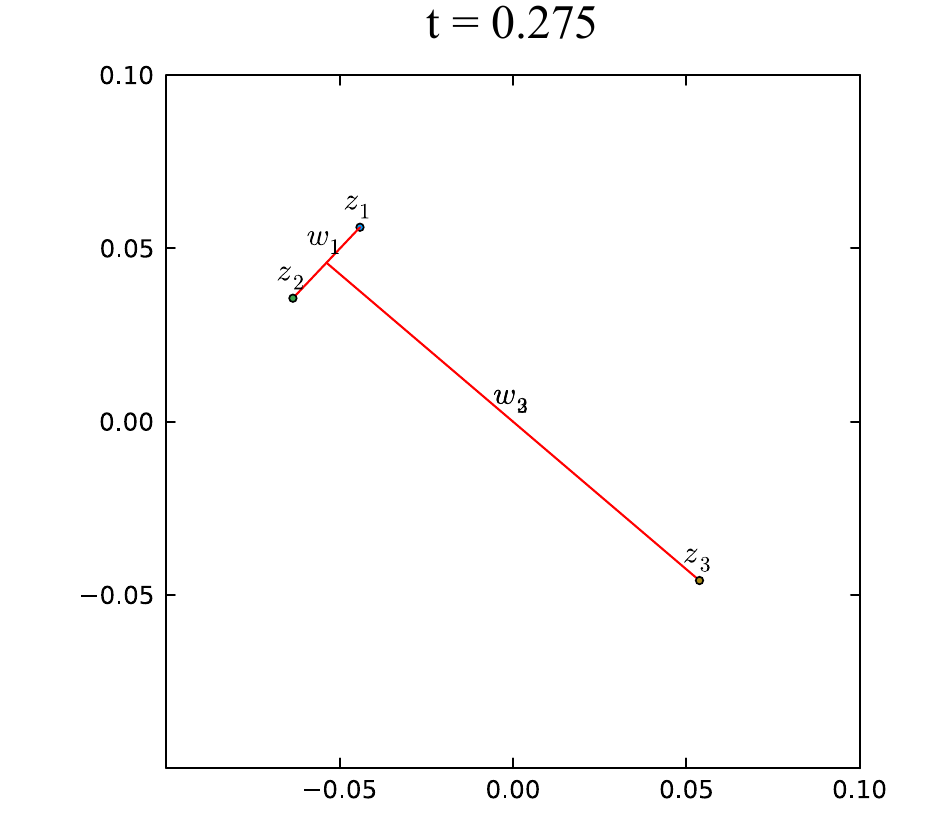}}
\par\end{raggedright}
\begin{raggedright}
\subfloat[\label{fig:3pt_jacobi_2}]{\includegraphics[scale=0.4]{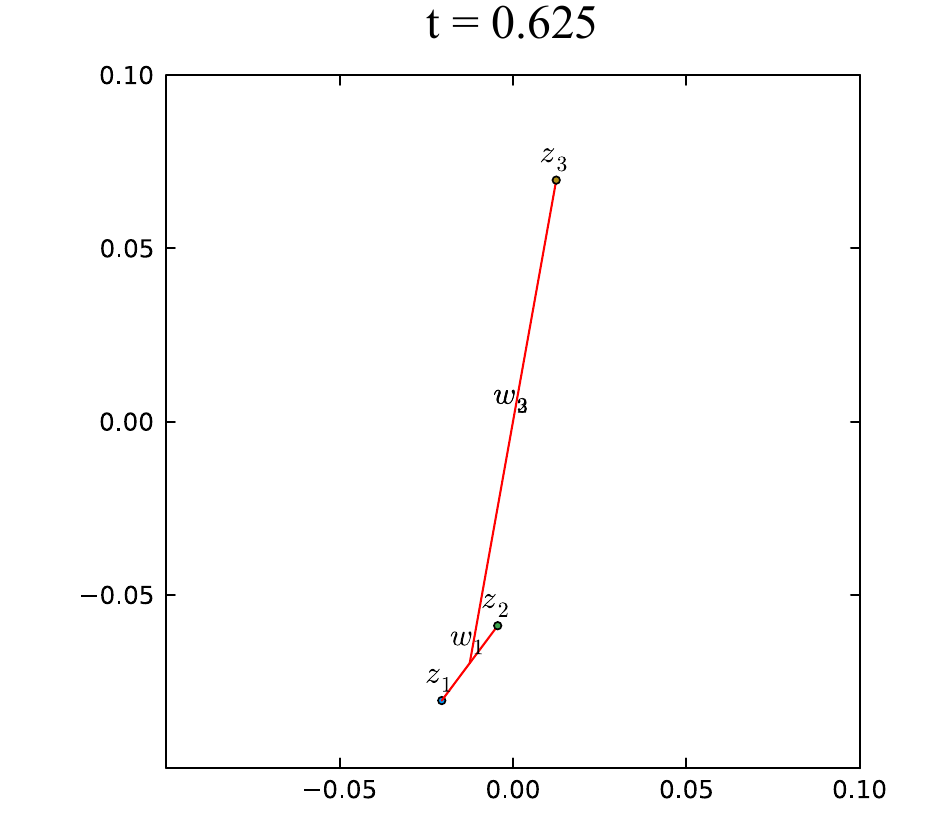}}\subfloat[\label{fig:3pt_full_time}]{\includegraphics[scale=0.4]{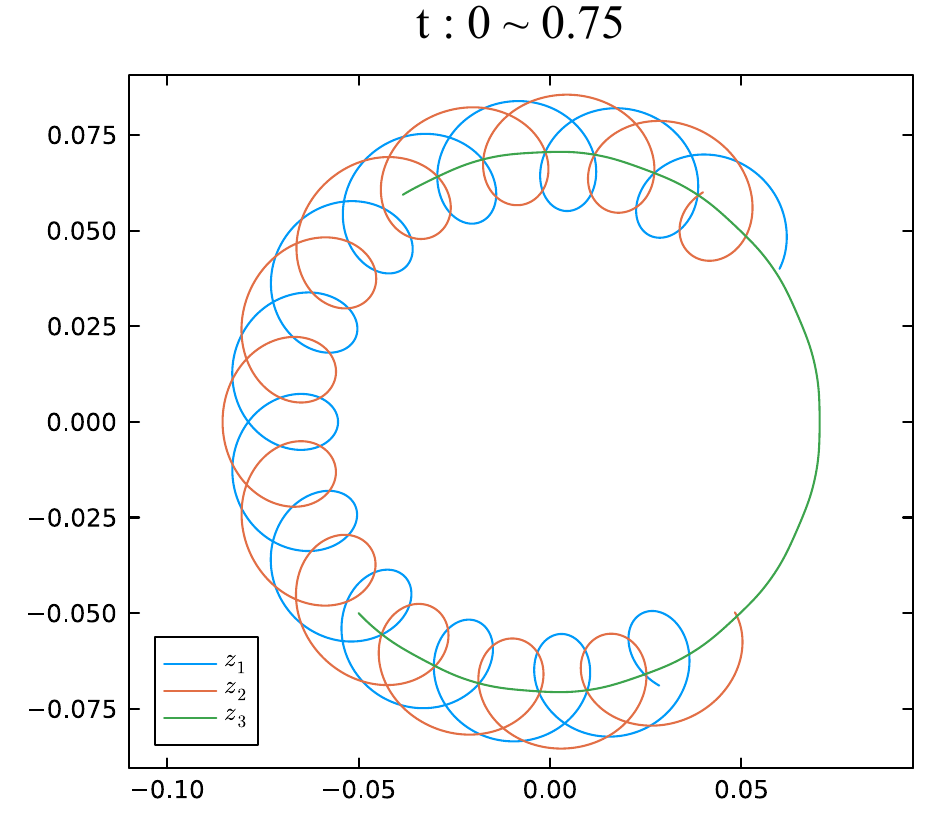}

}\caption{\label{fig:3pt_example}The orbits for $3$ point vortices on $\mathbb{R}^{2}$.
We set $a_{1}=a_{2}=1,a_{3}=2$ and $w_{1}\left(0\right)=\frac{\varepsilon^{2}}{2}-i\frac{\varepsilon^{2}}{2},w_{2}\left(0\right)=\frac{\varepsilon}{2}+i\frac{\varepsilon}{2},w_{3}\left(0\right)=0$.
\ref{fig:3pt_jacobi_0}, \ref{fig:3pt_jacobi_1}, and \ref{fig:3pt_jacobi_2}
are the instantaneous positions at different time and \ref{fig:3pt_full_time}
shows the full orbits for $0<t\le0.75$.}
\par\end{raggedright}
\end{figure}
\end{rem}
\begin{rem}
Without loss of generality, the assumption $w_{l}\left(0\right)\neq0$ can always be fulfilled.
Indeed if $w_{l}\left(0\right)=0$ for some non-adjacent $l$, then one can always
relabel the indices to ensure that this non-zero condition is satisfied.
If $w_{l}\left(0\right)=w_{l+1}\left(0\right)=0$,
it would imply $z_{l+1}=z_{l+2}$ which is impossible.
\end{rem}

\subsection{Notation}

Most of the notations used throughout this paper are summarized below.
\begin{itemize}
\item Point vortices strength: $a_{k}$. Total strength: $a=\sum_{k}a_{k}$.
Partial sum of vorticity strength: $A_{k}=\sum^{k}_{l=1}a_{l}$.
\item Point vortex position: $z_{k}=x_{k}+iy_{k}\simeq\left(x_{k},y_{k}\right)$.
\item The simply connected domain $\Omega\subset\mathbb{R}^{2}$.
\item The Green's function of $\Omega$: $G\left(z_{1},z_{2}\right)=\frac{1}{2\pi}\ln\left|z_{1}-z_{2}\right|+\gamma\left(z_{1},z_{2}\right)$,
where $\gamma$ is the regular part of $G$.
\item Robin's function: $\tilde{\gamma}\left(z\right)=\gamma\left(z,z\right)$.
\item The lowest order of a polynomial in $\left(\boldsymbol{\beta},\bar{\boldsymbol{\beta}}\right)$:
given $f=\sum_{\left(\boldsymbol{K},\boldsymbol{L}\right)}c_{\boldsymbol{K},\boldsymbol{L}}\boldsymbol{\beta}^{\boldsymbol{K}}\bar{\boldsymbol{\beta}}^{\boldsymbol{L}}$,
\[
\mathrm{ord}\left(f\right)=\min\left\{ \left|\boldsymbol{K}\right|+\left|\boldsymbol{L}\right|\big|c_{\boldsymbol{K},\boldsymbol{L}}\neq0\right\} .
\]
\item The set of all exponent pairs $(\boldsymbol{K},\boldsymbol{L})$
corresponding to the expansion of $f$:
\[
\mathscr{E}\left(f\right)=\left\{ \left(\boldsymbol{K},\boldsymbol{L}\right)\left|f=\sum_{\left(\boldsymbol{K},\boldsymbol{L}\right)}c_{\boldsymbol{K},\boldsymbol{L}}\boldsymbol{\beta}^{\boldsymbol{K}}\bar{\boldsymbol{\beta}}^{\boldsymbol{L}},c_{\boldsymbol{K},\boldsymbol{L}}\neq0\right.\right\} 
\]
\item The index set $\boldsymbol{J}$ is formed by concatenating any number of ranges $i$ to $j$ and individual indices. And for multi-index $\boldsymbol{K}$, $\boldsymbol{K}\left[\boldsymbol{J}\right]$ selects the components indexed by $\boldsymbol{J}$, identical to the order of indices in $\boldsymbol{J}$. For example, we can take $\boldsymbol{J} = \left[2:3,5,7:9\right]$ and 
\[
\boldsymbol{K}\left[\boldsymbol{J}\right] =\left(K_{2},K_{3},K_{5},K_{7},K_{8},K_{9}\right).
\]
The index set $\boldsymbol{S}_{\boldsymbol{K}}$ is called the support index set of the multi-index $\boldsymbol{K}$ if $\boldsymbol{K}\left[\boldsymbol{S}_{\boldsymbol{K}}\right]$ contains non-zero element and $\boldsymbol{K}\setminus\boldsymbol{K}\left[\boldsymbol{S}_{\boldsymbol{K}}\right]$ only includes zeros. For example, for multi-index $\boldsymbol{K}=\left(0, 1, 2, 0\right)$, its support index set is $\boldsymbol{S}_{\boldsymbol{K}}=\left[2:3\right]$. We also define the support of a multi-index by 
\[
supp\left(\boldsymbol{K}\right) = \boldsymbol{K}\left[\boldsymbol{S}_{\boldsymbol{K}}\right].
\]
We remark that one can define the support of a truncated multi-index in the same manner. For example, let $\boldsymbol{K}=(0, 1, 2, 0, 3, 4)$. Then the support of $\boldsymbol{K}\left[1:4\right]$ is 
\[
supp\left(\boldsymbol{K}\left[1:4\right]\right) = \left(1, 2\right)
\]
\item A multi-index $\boldsymbol{K}\in\mathbb{N}^{N}$. Let $\mathscr{R}_{0}$
be the set of all multi-index pairs $\left(\boldsymbol{K},\boldsymbol{L}\right)$.
Define the $m$-th level of resonant set
\[
\mathscr{R}_{m}=\left\{ \left(\boldsymbol{K},\boldsymbol{L}\right)\in\mathscr{R}_{0}\left|K_{1}=L_{1},\cdots,K_{m}=L_{m}\right.\right\} .
\]
We introduce another special set $\mathscr{C}$ by
\begin{align*}
\mathscr{C}	&=\bigg(\bigcup^{N}_{m=1}\left\{ \left(\boldsymbol{K},\boldsymbol{L}\right)\in\mathscr{R}_{m-1}\setminus\mathscr{R}_{m}\left|\begin{aligned} & supp\left(\boldsymbol{K}\left[1:m\right]\right)\text{ is connected}\\
 & K_{l}=L_{l}\neq1,1\le l\le m-1
\end{aligned}
\right.\right\} \\
	&\qquad\bigcup\left\{ \left(\boldsymbol{K},\boldsymbol{L}\right)\in\mathscr{R}_{N}\left|supp\left(\boldsymbol{K}\right)\text{ is connected}\right.\right\} \bigg)\\
	&\qquad\bigcap\left\{ \left(\boldsymbol{K},\boldsymbol{L}\right)\left|supp\left(\boldsymbol{K}+\boldsymbol{L}\right)\text{ is connected}\right.\right\} .
\end{align*}
For example, we take $N=4$ and 
\begin{align*}
    & \left(\boldsymbol{K}=\left(0,1,2,0\right),\;\boldsymbol{L}=\left(0,2,0,1\right)\right)\in\mathscr{C}, \quad m=2\\
    & \left(\boldsymbol{K}=\left(0,2,2,1\right),\;\boldsymbol{L}=\left(0,2,1,1\right)\right)\in\mathscr{C}, \quad m=3\\
    & \left(\boldsymbol{K}=\left(1,2,3,0\right),\;\boldsymbol{L}=\left(1,2,0,1\right)\right)\notin\mathscr{C}, \quad m=3, K_1=L_1=1. \\
    & \left(\boldsymbol{K}=\left(2,1,0,0\right),\;\boldsymbol{L}=\left(1,2,0,1\right)\right)\notin\mathscr{C},\quad \mbox{non-connected support of  } {\bf K}+{\bf L}.
\end{align*}
The set $\mathscr{C}$ is crucial in the construction of the normal form, as it remains invariant under the normal form procedure.
\end{itemize}

\section{Preliminaries}

\subsection{Review of the N Point Vortices system}

Let $\mathcal{N}$ be the phase space that includes all possible positions
of point vortices: 
\[
\mathcal{N}=\left\{ \left(x_{1},\cdots,x_{N},y_{1},\cdots,y_{N}\right)\big|\left(x_{k},y_{k}\right)\in\Omega,\left(x_{k},y_{k}\right)\neq\left(x_{l},y_{l}\right),1\le l<k\le N\right\} \subset\mathbb{R}^{2N}.
\]
Recall the the differential $2$-form on $\mathcal{N}$ is defined
as
\[
\omega_{\mathcal{N}}=\sum^{N}_{k=1}a_{k}dx_{k}\wedge dy_{k},
\]
Clearly $\omega_{\mathcal{N}}$ is a sympolectic form on $\mathcal{N}$.
It establishes an isomorphism from the cotangent space $T^{*}\mathcal{N}=\mathbb{R}^{2N}$
to tangent space $T\mathcal{N}=\mathbb{R}^{2N}$. Indeed the representation
theorem implies that each $\boldsymbol{\xi}\in T\mathcal{N}$ corresponds
to a $1$-form $\boldsymbol{\xi}^{*}\in T^{*}\mathcal{N}$ such that
$\boldsymbol{\xi}^{*}\left(\cdot\right)=\omega_{\mathcal{N}}\left(\boldsymbol{\xi},\cdot\right)$.
We denote the isomorphism by $J_{\mathcal{N}}:T^{*}\mathcal{N}\to T\mathcal{N}$.
For each $\boldsymbol{\eta}\in T\mathcal{N}$, it holds that 
\[
\boldsymbol{\xi}^{*}\left(\boldsymbol{\eta}\right)=\omega_{\mathcal{N}}\left(\boldsymbol{\xi},\boldsymbol{\eta}\right)=\sum^{N}_{k=1}a_{k}\left|\begin{array}{cc}
\xi_{k} & \xi_{k+N}\\
\eta_{k} & \eta_{k+N}
\end{array}\right|=\sum^{N}_{k=1}-a_{k}\xi_{k+N}\eta_{k}+a_{k}\xi_{k}\eta_{k+N}.
\]
It indicates that
\[
\boldsymbol{\xi}^{*}=\left(-a_{1}\xi_{N+1},\cdots,-a_{N}\xi_{2N},a_{1}\xi_{1},\cdots,a_{N}\xi_{N}\right)^{T}.
\]
Since $J_{\mathcal{N}}\boldsymbol{\xi}^{*}=\boldsymbol{\xi}$ we obtain
the matrix 
\[
J_{\mathcal{N}}=\left(\begin{array}{cccccc}
 &  &  & \frac{1}{a_{1}}\\
 & 0 &  &  & \ddots\\
 &  &  &  &  & \frac{1}{a_{N}}\\
-\frac{1}{a_{1}}\\
 & \ddots &  &  & 0\\
 &  & -\frac{1}{a_{N}}
\end{array}\right).
\]
The equations of motion for point vortices are represented as the
following Hamiltonian system
\[
\frac{d}{dt}\left(\begin{array}{c}
\boldsymbol{x}\\
\boldsymbol{y}
\end{array}\right)=J_{\mathcal{N}}d_{\mathcal{N}}H
\]
which are equivalent to 
\begin{equation}
\dot{x}_{k}=\frac{1}{a_{k}}\partial_{y_{k}}H,\dot{y}_{k}=-\frac{1}{a_{k}}\partial_{x_{k}}H,\label{eq:N_pt_vector}
\end{equation}
where the Hamiltonian reads 
\[
H=-\sum_{k=1}\sum_{l=1}\frac{a_{k}a_{l}}{2}\gamma\left(z_{k},z_{l}\right)-\sum_{k=1}\sum_{l<k}\frac{a_{k}a_{l}}{2\pi}\ln\left|z_{k}-z_{l}\right|.
\]
The Poisson bracket is defined by
\begin{align*}
\left\{ F,G\right\} _{\mathcal{N}} & =\omega_{\mathcal{N}}\left(J_{\mathcal{N}}d_{\mathcal{N}}F,J_{\mathcal{N}}d_{\mathcal{N}}G\right)\\
 & =\sum^{N}_{k=1}\frac{1}{a_{k}}\left(\partial_{x_{k}}F\partial_{y_{k}}G-\partial_{y_{k}}F\partial_{x_{k}}G\right).
\end{align*}
The vector field $J_{\mathcal{N}}d_{\mathcal{N}}H$ is also called
the Hamiltonian vector field generated by $H$. 

\subsubsection{Reformulation under a linear invertible map}

We now introduce the change of variables 
\[
\left(w_{1},w_{2},\cdots,w_{N}\right)\to\left(z_{1},z_{2},\cdots,z_{N}\right),\quad w_{k}=\left(\xi_{k},\eta_{k}\right),
\]
defined by a linear invertible map $\mathcal{L}:\mathcal{M}\to\mathcal{N}$
whose matrix is the following
\[
\mathcal{L}=\left(\begin{array}{cc}
L & 0\\
0 & L
\end{array}\right),
\]
where $L=\left(L_{mn}\right)_{1\le m\le N,1\le n\le N}$ is an invertible
$N\times N$ matrix. Thus a symplectic structure on $\mathcal{M}$
induced by $\omega_{\mathcal{N}}$ is 
\[
\omega_{\mathcal{M}}=\mathcal{L}^{*}\omega_{\mathcal{N}}=\sum^{N}_{k=1}\sum_{m,n}a_{k}L_{km}L_{kn}d\xi_{m}\wedge d\eta_{n}=\sum_{m,n}B_{mn}d\xi_{m}\wedge d\eta_{n}
\]
where $B_{mn}=\sum^{N}_{k=1}a_{k}L_{km}L_{kn}$. The isomorphism from
$T^{*}\mathcal{M}$ to $T\mathcal{M}$ is 
\[
J_{\mathcal{M}}=\left(\begin{array}{cc}
0 & -B\\
B & 0
\end{array}\right)^{-1}=\left(\begin{array}{cc}
0 & B^{-1}\\
-B^{-1} & 0
\end{array}\right)
\]
with $B=\left(B_{mn}\right)_{1\le m\le N,1\le n\le N}$ a symmetric
matrix. Letting $B^{-1}=\left(C_{mn}\right)_{1\le m\le N,1\le n\le N}$,
the Poisson bracket has the following expression
\begin{align*}
\left\{ F,G\right\} _{\mathcal{M}} & =\omega_{\mathcal{M}}\left(J_{\mathcal{M}}d_{\mathcal{M}}F,J_{\mathcal{M}}d_{\mathcal{M}}G\right)\\
 & =\sum_{m,n}B_{mn}\left|\begin{array}{cc}
\left(B^{-1}\nabla_{\eta}F\right)_{m} & -\left(B^{-1}\nabla_{\xi}F\right)_{n}\\
\left(B^{-1}\nabla_{\eta}G\right)_{m} & -\left(B^{-1}\nabla_{\xi}G\right)_{n}
\end{array}\right|\\
 & =\sum_{m,n}B_{mn}\left|\begin{array}{cc}
\sum_{l}C_{ml}\partial_{\eta_{l}}F & -\sum_{k}C_{nk}\partial_{\xi_{k}}F\\
\sum_{l}C_{ml}\partial_{\eta_{l}}G & -\sum_{k}C_{nk}\partial_{\xi_{k}}G
\end{array}\right|\\
 & =\sum_{m,n}\sum_{k,l}B_{mn}C_{ml}C_{nk}\left(\partial_{\xi_{k}}F\partial_{\eta_{l}}G-\partial_{\eta_{l}}F\partial_{\xi_{k}}G\right)\\
 & =\sum_{k,l}C_{kl}\left(\partial_{\xi_{k}}F\partial_{\eta_{l}}G-\partial_{\eta_{l}}F\partial_{\xi_{k}}G\right),
\end{align*}
where we use the fact that $B^{-1}$ is also symmetric so that $\sum_{m,n}B_{mn}C_{ml}C_{nk}=C_{kl}$
at the last step. The dynamic equations under coordinates $\left(\xi,\eta\right)$
are 
\begin{equation}
\dot{\xi}_{k}=\left\{ \xi_{k},H\right\} _{\mathcal{M}},\dot{\eta}_{k}=\left\{ \eta_{k},H\right\} _{\mathcal{M}},k=1,2,\cdots,N.\label{eq:dynamic_eqn_pq}
\end{equation}

\subsubsection{Jacobi coordinates and its modification}

Recall that $A_{k}=\sum^{k}_{l=1}a_{l}$ and are all non-zero. The
classical Jacobi coordinates:
\begin{align*}
w_{k} & =\frac{\sum^{k}_{l=1}a_{l}z_{l}}{A_{k}}-z_{k+1},k=1,2,\cdots,N-1,\\
w_{N} & =\frac{\sum^{N}_{l=1}a_{l}z_{l}}{A_{N}}.
\end{align*}
Let $\kappa_{k}=\frac{a_{k}}{A_{k}},k=1,2,\cdots,N$, then the matrix
mapping $\boldsymbol{w}$ to $\boldsymbol{z}$ is 
\[
L=\left(\begin{array}{ccccc}
\kappa_{2} & \kappa_{3} & \cdots & \kappa_{N} & 1\\
-\left(1-\kappa_{2}\right) & \kappa_{3} & \cdots & \kappa_{N} & 1\\
0 & -\left(1-\kappa_{3}\right) & \cdots & \kappa_{N} & 1\\
\vdots &  & \ddots &  & \vdots\\
0 & \cdots & 0 & -\left(1-\kappa_{N}\right) & 1
\end{array}\right).
\]
Noticing that $\kappa_{1}=1$, $z_{k}$ can be recovered explicitly
by
\[
z_{k}=\left(\kappa_{k}-1\right)w_{k-1}+\sum^{N}_{j=k}\kappa_{j+1}w_{j},
\]
where the sum index $j$ is understood as an element in $\mathbb{Z}/N\mathbb{Z}$.
Assuming that $D=\text{diag}\left(a_{1},a_{2},\cdots,a_{N}\right)$,
one can explicitly calculate that 
\begin{align*}
B & =L^{T}DL=\text{diag}\left(A_{1}\kappa_{2},A_{2}\kappa_{3},\cdots,A_{N-1}\kappa_{N},A_{N}\right),\\
C & =B^{-1}=\text{diag}\left(\frac{1}{A_{1}\kappa_{2}},\frac{1}{A_{2}\kappa_{3}},\cdots,\frac{1}{A_{N-1}\kappa_{N}},\frac{1}{A_{N}}\right).
\end{align*}
Therefore the sympolectic structure on $\mathcal{M}$ is 
\[
\omega_{\mathcal{M}}=\sum^{N}_{k=1}A_{k}\kappa_{k+1}d\xi_{k}\wedge d\eta_{k},
\]
and Poisson bracket becomes 
\[
\left\{ F,G\right\} _{\mathcal{M}}=\sum^{N}_{k=1}\frac{1}{A_{k}\kappa_{k+1}}\left(\partial_{\xi_{k}}F\partial_{\eta_{k}}G-\partial_{\eta_{k}}F\partial_{\xi_{k}}G\right).
\]
The Hamiltonian in Jacobi coordinates is 
\begin{align*}
H_{\Omega} & =-\sum_{k=1}\sum_{l<k}\frac{a_{k}a_{l}}{2\pi}\ln\left|\left(\kappa_{l}-1\right)w_{l-1}+\sum^{k-2}_{j=l}\kappa_{j+1}w_{j}+w_{k-1}\right|\\
 & \quad-\sum_{k=1}\sum_{l=1}\frac{a_{k}a_{l}}{2}\gamma\left(-w_{k-1}+\sum^{N}_{j=k-1}\kappa_{j+1}w_{j},-w_{l-1}+\sum^{N}_{j=l-1}\kappa_{j+1}w_{j}\right).
\end{align*}
It is worth noting that by expanding $\gamma$ (ref.~\cite{shen2026twopoint}) it holds that
\[
\dot{w}_{N}=A_{N}\left(\begin{array}{cc}
\frac{3}{2}\Im\left(c_{0}\right) & \frac{3}{2}\Re\left(c_{0}\right)-c_{1}\\
\frac{3}{2}\Re\left(c_{0}\right)+c_{1} & -\frac{3}{2}\Im\left(c_{0}\right)
\end{array}\right)w_{N}+h.o.t,
\]
where 
\[
c_{0}=\frac{\varphi'''\left(0\right)}{6\pi\varphi'\left(0\right)},c_{1}=\frac{\left|\varphi'\left(0\right)\right|^{2}}{2\pi}.
\]
The linear part of $\dot{w}_{N}$ is a rotation if the stability condition
$2\left|\varphi^{\prime}\left(0\right)\right|^{3}>\left|\varphi^{\prime\prime\prime}\left(0\right)\right|$
is satisfied. Then there is a linear map $L'$ with $\det\left(L'\right)=1$
that only maps $w_{N}$ to its standard form and leaves $w_{1},w_{2},\cdots,w_{N-1}$
unchanged. Denoting the new coordinates by $\zeta_{k}=\left(p_{k},q_{k}\right)\simeq p_{k}+iq_{k}$
and the phase space by $\mathcal{S}$, the corresponding sympolectic
structure becomes
\[
\omega_{\mathcal{S}}=\sum^{N-1}_{k=1}A_{k}\kappa_{k+1}dp_{k}\wedge dq_{k}+A_{N}\det\left(L'\right)dp_{N}\wedge dq_{N}=\sum^{N}_{k=1}A_{k}\kappa_{k+1}dp_{k}\wedge dq_{k}.
\]
And the Poisson bracket, Hamiltonian are the following
\begin{align}
\left\{ F,G\right\} _{\mathcal{S}} & =\sum^{N}_{k=1}\frac{1}{A_{k}\kappa_{k+1}}\left(\partial_{p_{k}}F\partial_{q_{k}}G-\partial_{q_{k}}F\partial_{p_{k}}G\right).\label{eq:poisson_bracket}\\
H & =-\sum_{k=1}\sum_{l<k}\frac{a_{k}a_{l}}{2\pi}\ln\left|\left(\kappa_{l}-1\right)\zeta_{l-1}+\sum^{k-2}_{j=l}\kappa_{j+1}\zeta_{j}+\zeta_{k-1}\right|\nonumber \\
 & \quad-\sum_{k=1}\sum_{l=1}\frac{a_{k}a_{l}}{2}\gamma\left(S^{-1}\zeta_{N}-\zeta_{k-1}+\sum^{N-1}_{j=k-1}\kappa_{j+1}\zeta_{j},S^{-1}\zeta_{N}-\zeta_{l-1}+\sum^{N-1}_{j=l-1}\kappa_{j+1}\zeta_{j}\right).\label{eq:hamiltonian}
\end{align}
We will work on the coordinates $\zeta_{k}$.

\subsection{Lie Transform}

Given a smooth function $g:\mathcal{S}\to\mathcal{S}$, its Hamiltonian
vector field is $\boldsymbol{X}_{g}=\boldsymbol{J}_{\mathcal{\mathcal{S}}}d_{\mathcal{\mathcal{S}}}g$.
Let $\phi^{s}_{g}$ be the Hamiltonian phase flow
\[
\frac{d}{ds}\phi^{s}_{g}\left(\boldsymbol{\zeta}\right)=\boldsymbol{X}_{g}\left(\phi^{s}_{g}\left(\boldsymbol{\zeta}\right)\right).
\]
The Lie transform is defined by the time $1$-map $\phi_{g}:=\phi^{1}_{g}$.
Then $\phi_{g}$ preserves the differential $2$-form $\omega_{\mathcal{\mathcal{S}}}$.
Thus the change of variable 
\[
\left(p,q\right)=\phi_{g}\left(P,Q\right)
\]
preserves the structure of dynamic equations \eqref{eq:dynamic_eqn_pq}
and 
\[
\dot{P}_{k}=\left\{ P_{k},H\circ\phi_{g}\right\} _{\mathcal{S}},\dot{Q}_{k}=\left\{ Q_{k},H\circ\phi_{g}\right\} _{\mathcal{S}}.
\]
Therefore, it is enough to consider the Hamiltonian under Lie transform.
We remark that by definition of the Poisson bracket, we have 
\[
\frac{d}{ds}H\circ\phi^{s}_{g}=\left\{ g,H\right\} _{\mathcal{S}}\circ\phi^{s}_{g}:=\mathcal{L}_{g}\left(H\right)\circ\phi^{s}_{g},
\]
where $\mathcal{L}_{g}\left(\cdot\right):=\left\{ g,\cdot\right\} _{\mathcal{S}}$
is called the Lie derivative along the Hamiltonian vector field generated
by $g$. More generally, taking derivative $l$ times yields
\[
\frac{d^{l}}{ds^{l}}H\circ\phi^{s}_{g}=\mathcal{L}^{l}_{g}\left(H\right)\circ\phi^{s}_{g}.
\]
Then its Maclaurin expansion of $H\circ\phi^{s}_{g}$ is 
\[
H\circ\phi^{s}_{g}=\sum^{\infty}_{l=0}s^{l}H_{l}=\sum^{\infty}_{l=0}\frac{s^{l}}{l!}\mathcal{L}^{l}_{g}\left(H\right).
\]
Therefore evaluating at $s=1$ gives 
\begin{equation}
H\circ\phi_{g}=\sum^{\infty}_{l=0}\frac{1}{l!}\mathcal{L}^{l}_{g}\left(H\right).\label{eq:transform_expand}
\end{equation}

\subsection{Normal Form Preparation}

Recall the Hamiltonian \eqref{eq:hamiltonian}{\small
\begin{align*}
H_{\Omega} & =-\sum_{k=1}\sum_{l<k}\frac{a_{k}a_{l}}{2\pi}\ln\left|\left(\kappa_{l}-1\right)\zeta_{l-1}+\sum^{k-2}_{j=l}\kappa_{j+1}\zeta_{j}+\zeta_{k-1}\right|\\
 & \quad-\sum_{k=1}\sum_{l=1}\frac{a_{k}a_{l}}{2}\gamma\left(S^{-1}\zeta_{N}-\zeta_{k-1}+\sum^{N-1}_{j=k-1}\kappa_{j+1}\zeta_{j},S^{-1}\zeta_{N}-\zeta_{l-1}+\sum^{N-1}_{j=l-1}\kappa_{j+1}\zeta_{j}\right),
\end{align*}
}
where the regular part $\gamma$  can be explicitly represented as 
\begin{align}
    \gamma\left(z_1,z_2\right)	&=-\frac{1}{2\pi}\ln\left|\left(1-\varphi\left(z_1\right)\overline{\varphi\left(z_2\right)}\right)\left(1-\overline{\varphi\left(z_1\right)}\varphi\left(z_2\right)\right)\right| \nonumber \\ 
	&\qquad+\frac{1}{2\pi}\ln\frac{\left|\varphi\left(z_1\right)-\varphi\left(z_2\right)\right|}{\left|z_1-z_2\right|},\quad z_1,z_2\in\mathbb{C}. \label{eq:gamma}
\end{align}
Thanks to Lemma 2.1 in~\cite{shen2026twopoint}, we know the expansion of $\gamma$
up to the second order: for $x,y$ small
\begin{align}
\gamma\left(z_1,z_2\right) & =\frac{1}{2\pi}\ln\left|\varphi'\left(0\right)\right|+\frac{1}{2}\Re\left(c_{0}\left(z+1^{2}+z_1z+2+z_2^{2}\right)\right)+c_{1}\Re\left(z_1\bar{z}_2\right)+h.o.t,\label{eq:gamma_expand}\\
\tilde{\gamma}\left(z\right) & =\frac{1}{2\pi}\ln\left|\varphi'\left(0\right)\right|+\frac{3}{2}\Re\left(c_{0}z^{2}\right)+c_{1}\left|z\right|^{2}+h.o.t.\label{eq:tilde_gamma_expand}
\end{align}
where 
\[
c_{0}=\frac{\varphi'''\left(0\right)}{6\pi\varphi'\left(0\right)},c_{1}=\frac{\left|\varphi'\left(0\right)\right|^{2}}{2\pi}.
\]
Expanding $\gamma$ by \eqref{eq:gamma_expand} and letting 
\begin{align*}
z_1 & =S^{-1}\zeta_{N}-\zeta_{k-1}+\sum^{N-1}_{j=k-1}\kappa_{j+1}\zeta_{j},\\
z_2 & =S^{-1}\zeta_{N}-\zeta_{l-1}+\sum^{N-1}_{j=l-1}\kappa_{j+1}\zeta_{j}
\end{align*}
 imply 
\begin{align*}
H_{\Omega} & =-\sum_{k=1}\sum_{l<k}\frac{a_{k}a_{l}}{2\pi}\ln\left|\left(\kappa_{l}-1\right)\zeta_{l-1}+\sum^{k-2}_{j=l}\kappa_{j+1}\zeta_{j}+\zeta_{k-1}\right|\\
 & \quad-\frac{1}{2}\omega\left|\zeta_{N}\right|^{2}-\sum^{N-1}_{k=1}\frac{\kappa_{k+1}A_{k}A_{n}}{4}\left(c_{0}\zeta^{2}_{k}+\bar{c}_{0}\bar{\zeta}^{2}_{k}\right)+h.o.t.
\end{align*}
By introducing the normalized variable $\boldsymbol{\alpha}=\frac{\boldsymbol{\zeta}}{\sqrt{2}},\bar{\boldsymbol{\alpha}}=\frac{\bar{\boldsymbol{\zeta}}}{\sqrt{2}}$,
the $l$-th action variable is simply $I_{l}=\alpha_{l}\bar{\alpha}_{l}$.
The above Hamiltonian then becomes 
\begin{align}
H_\Omega & =\underbrace{-\sum^{N-1}_{k=1}\frac{a_{k+1}A_{k}}{4\pi}\ln\left(\alpha_{k}\bar{\alpha}_{k}\right)-\omega\alpha_{N}\bar{\alpha}_{N}}_{h_{0}}\nonumber \\
 & \quad-\sum_{k=1}\sum_{l<k}\frac{a_{k}a_{l}}{2\pi}\ln\left|1+\left(\kappa_{l}-1\right)\frac{\alpha_{l-1}}{\alpha_{k-1}}+\sum^{k-2}_{j=l}\kappa_{j+1}\frac{\alpha_{j}}{\alpha_{k-1}}\right|\nonumber \\
 & \quad-\sum^{N-1}_{k=1}\frac{\kappa_{k+1}A_{k}A_{N}}{2}\left(c_{0}\alpha^{2}_{k}+\bar{c}_{0}\bar{\alpha}^{2}_{k}\right)+h.o.t,\label{eq:Ham}
\end{align}
where $\omega>0$ is the characteristic frequency of the linear motion
of the center of vorticity. We expect that $h_{0}$ is the dominant
part and consider the rest as the perturbation. Apparently, the perturbation
$H-h_{0}$ is not standard since it is not a power series in terms
of $\left(\boldsymbol{\alpha},\bar{\boldsymbol{\alpha}}\right)$.
we define the variables $\boldsymbol{\beta}=\left(\beta_{1},\beta_{2},\cdots,\beta_{N}\right)$
to help doing the normal form procedure:
\begin{align*}
 & \beta_{i}:=\frac{\alpha_{i}}{\alpha_{i+1}},i=1,2,\cdots,N-2,\\
 & \beta_{N-1}:=\alpha_{N-1},\beta_{N}:=\alpha_{N}.
\end{align*}
By noting that for $1\le l\le k\le N-1$,
\[
\frac{\alpha_{l}}{\alpha_{k}}=\prod^{k-1}_{i=l}\beta_{i},\quad\alpha_{l}=\prod^{N-1}_{i=l}\beta_{i}
\]
so given all $\beta_{i}$ small, one can expand the Hamiltonian as
the following{\footnotesize
\begin{align}
H_{\Omega} & =h_{0}-\sum^{N}_{k=1}\sum^{k-1}_{l=2}\frac{a_{k}a_{l}}{2\pi}\ln\left|1+\left(\kappa_{l}-1\right)\beta_{l-1}\beta_{l}\cdots\beta_{k-2}+\sum^{k-2}_{j=l}\kappa_{j+1}\beta_{j}\beta_{j+1}\cdots\beta_{k-2}\right|\nonumber \\
 & \quad-\sum^{N-1}_{k=1}\frac{\kappa_{k+1}A_{k}A_{N}}{2}\left(c_{0}\left(\beta_{k}\beta_{k+1}\cdots\beta_{N-1}\right)^{2}+\bar{c}_{0}\left(\bar{\beta}_{k}\bar{\beta}_{k+1}\cdots\bar{\beta}_{N-1}\right)^{2}\right)+h.o.t\nonumber \\
 & =h_{0}+\sum_{\left|\boldsymbol{K}\right|+\left|\boldsymbol{L}\right|\ge1}c_{0,K,L}\beta^{\boldsymbol{K}}\bar{\beta}^{\boldsymbol{L}}\nonumber \\
 & :=h_{0}+R_{0}.\label{eq:H0}
\end{align}
}A key observation for $R_{0}$ is that the degree pairs $(\boldsymbol{K},\boldsymbol{L})$
are highly constrained. Indeed one can show that $\mathscr{E}\left(R_0\right)\in\mathscr{C}$ (see details in the proof of Lemma \ref{lem:normal_form_formal}). 

Next we consider the Poisson bracket \eqref{eq:poisson_bracket}.
In $\left(\boldsymbol{\alpha},\bar{\boldsymbol{\alpha}}\right)$ coordinates
the Poisson bracket has form
\[
\left\{ F,G\right\} _{\mathcal{S}}=\sum^{N}_{k=1}\frac{i}{A_{k}\kappa_{k+1}}\left(\partial_{\alpha_{k}}F\partial_{\bar{\alpha}_{k}}G-\partial_{\bar{\alpha}_{k}}F\partial_{\alpha_{k}}G\right).
\]
where the index $k\in\mathbb{Z}/N\mathbb{Z}$. We are interested in
computing $\left\{ f,h_{0}\right\} _{\mathcal{S}}$ and $\left\{ f,g\right\} _{\mathcal{S}}$
where $f,g$ are polynomials in $\left(\boldsymbol{\beta},\bar{\boldsymbol{\beta}}\right)$.
It naturally reduces to evaluating the bracket of monomials, namely
$\left\{ \boldsymbol{\beta}^{\boldsymbol{K}}\bar{\boldsymbol{\beta}}^{\boldsymbol{L}},h_{0}\right\} _{\mathcal{S}}$
and $\left\{ \boldsymbol{\beta}^{\boldsymbol{K}}\bar{\boldsymbol{\beta}}^{\boldsymbol{L}},\boldsymbol{\beta}^{\boldsymbol{K}'}\bar{\boldsymbol{\beta}}^{\boldsymbol{L}'}\right\} _{\mathcal{S}}$.
A straightforward calculation shows {\scriptsize
\begin{align}
\left\{ \boldsymbol{\beta}^{\boldsymbol{K}}\bar{\boldsymbol{\beta}}^{\boldsymbol{L}},h_{0}\right\} _{\mathcal{S}} & =-\frac{i}{4\pi}\sum^{N-1}_{l=1}A_{l+1}\left(\delta_{l}\left(\boldsymbol{K}\right)-\delta_{l}\left(\boldsymbol{L}\right)\right)\boldsymbol{\beta}^{\boldsymbol{K}-\boldsymbol{E}_{l}}\bar{\boldsymbol{\beta}}^{\boldsymbol{L}-\boldsymbol{E}_{l}}\nonumber \\
 & \quad-i\frac{\omega}{A_{N}}\left(\delta_{N}\left(\boldsymbol{K}\right)-\delta_{N}\left(\boldsymbol{L}\right)\right)\boldsymbol{\beta}^{\boldsymbol{K}}\boldsymbol{\beta}^{\boldsymbol{L}},\label{eq:h0_action}\\
\left\{ \boldsymbol{\beta}^{\boldsymbol{K}}\bar{\boldsymbol{\beta}}^{\boldsymbol{L}},\boldsymbol{\beta}^{\boldsymbol{K}'}\bar{\boldsymbol{\beta}}^{\boldsymbol{L}'}\right\} _{\mathcal{S}} & =i\sum^{N-1}_{l=1}\frac{\delta_{l}\left(\boldsymbol{K}\right)\delta_{l}\left(\boldsymbol{L}'\right)-\delta_{l}\left(\boldsymbol{L}\right)\delta_{l}\left(\boldsymbol{K}'\right)}{A_{l}\kappa_{l+1}}\boldsymbol{\beta}^{\boldsymbol{K}+\boldsymbol{K}'-\boldsymbol{E}_{l}}\bar{\boldsymbol{\beta}}^{\boldsymbol{L}+\boldsymbol{L}'-\boldsymbol{E}_{l}}\nonumber \\
 & \quad-i\frac{\delta_{N}\left(\boldsymbol{K}\right)\delta_{N}\left(\boldsymbol{L}'\right)-\delta_{N}\left(\boldsymbol{L}\right)\delta_{N}\left(\boldsymbol{K}'\right)}{A_{N}}\boldsymbol{\beta}^{\boldsymbol{K}+\boldsymbol{K}'-\boldsymbol{e}_{N}}\bar{\boldsymbol{\beta}}^{\boldsymbol{L}+\boldsymbol{L}'-\boldsymbol{e}_{N}},\label{eq:poly_action}
\end{align}
}where $\boldsymbol{E}_{m}$ is an index translation vector: 
\[
\left(\boldsymbol{E}_{m}\right)_{j}=\begin{cases}
1, & m\le j\le N-1,\\
0, & \text{otherwise},
\end{cases}
\]
and for a given multi-index $\boldsymbol{K}$ the function $\delta_{j}\left(\boldsymbol{K}\right)$
is defined by
\[
\delta_{j}\left(\boldsymbol{K}\right)=\begin{cases}
K_{1}, & j=1,\\
K_{j}-K_{j-1}, & 2\le j\le N-1,\\
K_{N}, & j=N.
\end{cases}
\]

\section{A formal normal form construction}

In this section, we perform the formal Birkhoff normal form up to
degree $n$. We agree that any reference to the degree or order of
a polynomial are all with respect to the variables $\left(\boldsymbol{\beta},\bar{\boldsymbol{\beta}}\right)$.
Recall that we consider 
\[
h_{0}=-\sum^{N-1}_{k=1}\frac{a_{k+1}A_{k}}{4\pi}\ln\left(\alpha_{k}\bar{\alpha}_{k}\right)-\omega\alpha_{N}\bar{\alpha}_{N}
\]
as the unperturbed Hamiltonian. Identity \eqref{eq:poly_action} shows that for two polynomials in $\left(\boldsymbol{\beta},\bar{\boldsymbol{\beta}}\right)$ with degree $k,l$, the maximum degree loss in Poisson bracket is $2N-2$ (corresponding to $l=1$ in \eqref{eq:poly_action}) while in the classical case, the loss is only $2$. This is a serious obstruction to run a normal form procedure even in a formal sense. In this section we will show the following formal
result. Recall that 
\[
\mathscr{E}\left(f\right)=\left\{ \left(\boldsymbol{K},\boldsymbol{L}\right)\left|f=\sum_{\left(\boldsymbol{K},\boldsymbol{L}\right)}c_{\boldsymbol{K},\boldsymbol{L}}\boldsymbol{\beta}^{\boldsymbol{K}}\bar{\boldsymbol{\beta}}^{\boldsymbol{L}},c_{\boldsymbol{K},\boldsymbol{L}}\neq0\right.\right\} 
\]
is the set of all exponent pairs $(\boldsymbol{K},\boldsymbol{L})$
corresponding to the expansion of $f$. 
\begin{lem}
\label{lem:normal_form_formal}Given $n=1,2,3,
\cdots$, there
exists a canonical transformation $T_{n}\left(\boldsymbol{\alpha},\bar{\boldsymbol{\alpha}}\right)$
formally such that 
\[
H_{n}:=H_\Omega\circ T_{n}=h_{0}+Z_{n}+R_{n},
\]
where $\left\{ Z_{n},h_{0}\right\} _{\mathcal{S}}=0$ and in terms
of $\left(\boldsymbol{\beta},\bar{\boldsymbol{\beta}}\right)$, it
holds that $\mathrm{ord}\left(R_{n}\right)\ge n+1$. Moreover, all
the exponents of $Z_{n},R_{n}$ satisfy
\[
\mathscr{E}\left(Z_{n}\right)\subset\mathscr{C},\quad\mathscr{E}\left(R_{n}\right)\subset\mathscr{C}.
\]
\end{lem}
\begin{proof}
The construction goes recursively. Let 
\[
H_{0}:=H_\Omega=h_{0}+Z_{0}+R_{0}.
\]
Clearly $Z_{0}=0$ and according to \eqref{eq:H0}, $\mathrm{ord}\left(R_{0}\right)=1$.
Next we show $\mathscr{E}\left(R_0\right)\subset\mathscr{C}$. From \eqref{eq:H0}, we know
{\small
\begin{align}
R_{0} & =-\sum^{N}_{k=1}\sum^{k-1}_{l=2}\frac{a_{k}a_{l}}{2\pi}\ln\left|1+\left(\kappa_{l}-1\right)\beta_{l-1}\beta_{l}\cdots\beta_{k-2}+\sum^{k-2}_{j=l}\kappa_{j+1}\beta_{j}\beta_{j+1}\cdots\beta_{k-2}\right|\label{eq:from_plane}\\
 & \quad-\sum^{N-1}_{k=1}\frac{\kappa_{k+1}A_{k}A_{N}}{2}\left(c_{0}\left(\beta_{k}\beta_{k+1}\cdots\beta_{N-1}\right)^{2}+\bar{c}_{0}\left(\bar{\beta}_{k}\bar{\beta}_{k+1}\cdots\bar{\beta}_{N-1}\right)^{2}\right)+h.o.t.\label{eq:boundary_effect}
\end{align}
}
Appearently, by considering $\boldsymbol{\beta},\bar{\boldsymbol{\beta}}$
as independed variables, $\boldsymbol{\beta}^{\boldsymbol{K}}\bar{\boldsymbol{\beta}}^{\boldsymbol{L}}$
in the expansion of \eqref{eq:from_plane} satisfies either $\boldsymbol{K}=\boldsymbol{0}$
or $\boldsymbol{L}=\boldsymbol{0}$. In addition, the support
of non-zero multi-index $\boldsymbol{K}$ or $\boldsymbol{L}$ is
connected, so that $\mathscr{E}\left(\eqref{eq:from_plane}\right)\subset\mathscr{C}$. For \eqref{eq:boundary_effect}, it is derived from 
\[
\sum^{N}_{k=1}\sum^{N}_{l=1}\frac{a_{k}a_{l}}{2}\gamma\left(z_{k},z_{l}\right)
\]
through series of changes of variables
\[
\boldsymbol{z}\to\boldsymbol{w}\to\boldsymbol{\zeta}\to\boldsymbol{\alpha}\to\boldsymbol{\beta}.
\]
Noticing that from \eqref{eq:gamma}, it holds that 
\begin{align}
    & \sum^{N}_{k=1}\sum^{N}_{l=1}\frac{a_k a_l}{2}\gamma\left(z_k,z_l\right) \nonumber \\	
    =&-\frac{1}{4\pi}\sum^{N}_{k=1}\sum^{N}_{l=1}a_k a_l\ln\left|\left(1-\varphi\left(z_k\right)\overline{\varphi\left(z_l\right)}\right)\left(1-\overline{\varphi\left(z_k\right)}\varphi\left(z_l\right)\right)\right| \label{eq:mixed}\\ 
	&+\frac{1}{4\pi}\sum^{N}_{k=1}\sum^{N}_{l=1}a_k a_l\ln\frac{\left|\varphi\left(z_k\right)-\varphi\left(z_l\right)\right|}{\left|z_k-z_l\right|}. \label{eq:nonmixed}
\end{align}
For \eqref{eq:nonmixed}, the expansion in terms of $\boldsymbol{\beta},\bar{\boldsymbol{\beta}}$ has the same structure as in \eqref{eq:from_plane}. Thus one has $\mathscr{E}\left(\eqref{eq:nonmixed}\right)\subset\mathscr{C}$. To prove $\mathscr{E}\left(\eqref{eq:mixed}\right)\subset\mathscr{C}$, it is sufficient to show
\begin{equation}\label{eq:mix_middle_step}
\mathscr{E}\left(\sum^{N}_{k=1}\sum^{N}_{l=1}a_{k}a_{l}\ln\left(1-\varphi\left(z_{k}\right)\overline{\varphi\left(z_{l}\right)}\right)\right)\subset\mathscr{C}.
\end{equation}
Noticing that we have the expansion
\begin{align*}
\sum^{N}_{k=1}\sum^{N}_{l=1}a_{k}a_{l}\ln\left(1-\varphi\left(z_{k}\right)\overline{\varphi\left(z_{l}\right)}\right) & =\sum_{p\ge1}\sum_{k,l}-\frac{a_{k}a_{l}}{p}\left(\varphi\left(z_{k}\right)\overline{\varphi\left(z_{l}\right)}\right)^{p}\\
 & =\sum_{p\ge1}-\frac{1}{p}\sum^{N}_{k=1}a_{k}\varphi\left(z_{k}\right)^{p}\sum^{N}_{l=1}a_{l}\overline{\varphi\left(z_{l}\right)}^{p},
\end{align*}
applying Proposition \ref{prop:mix_structure} gives the following power series in $\boldsymbol{w},\bar{\boldsymbol{w}}$
\[
\sum^{N}_{k=1}\sum^{N}_{l=1}a_{k}a_{l}\ln\left(1-\varphi\left(z_{k}\right)\overline{\varphi\left(z_{l}\right)}\right)=\sum c_{\boldsymbol{K},\boldsymbol{L}}\boldsymbol{w}^{\boldsymbol{K}}\bar{\boldsymbol{w}}^{\boldsymbol{L}}
\]
where for any non-vanishing term $c_{\boldsymbol{K},\boldsymbol{L}} \neq 0$, the first non-zero component $K_i$ of $\boldsymbol{K}$ satisfies $K_i \ge 2$ whenever $i \le N-1$, and similarly, the first non-zero component $L_j$ of $\boldsymbol{L}$ satisfies $L_j \ge 2$ whenever $j \le N-1$. Plugging in the change of variable $\boldsymbol{w}\to\boldsymbol{\beta}$ proves \eqref{eq:mix_middle_step}.
In summary, the Lemma is true for $H_0$. 

Assuming that 
\[
H_{k-1}=h_{0}+Z_{k-1}+R_{k-1}
\]
satisfies $\left\{ Z_{k-1},h_{0}\right\} _{\mathcal{S}}=0, \mathrm{ord}\left(R_{k-1}\right)=k$ and $\mathscr{E}\left(R_{k-1}\right)\subset \mathscr{C}, \mathscr{E}\left(Z_{k-1}\right)\subset \mathscr{C}$.
Let $p_{k-1}$ be the terms of degree $k$ in $R_{k-1}$. To the next
step, the canonical transformation $\phi_{k}$ shall satisfy
\[
H_{k}=H_{k-1}\circ\phi_{k}=h_{0}+Z_{k}+\sum_{\left|\boldsymbol{K}\right|+\left|\boldsymbol{L}\right|\ge k+1}c_{k,\boldsymbol{K},\boldsymbol{L}}\boldsymbol{\beta}^{\boldsymbol{K}}\bar{\boldsymbol{\beta}}^{\boldsymbol{L}}.
\]
Assuming the generating function of $\phi_{k}$ is $g_{k}$. According
to \eqref{eq:transform_expand}, it holds that
\begin{align}
H_{k} & =h_{0}+Z_{k-1}\nonumber \\
 & \quad+\left\{ g_{k},h_{0}\right\}_{\mathcal{S}} +p_{k-1}\label{eq:linear_eqn}\\
 & \quad+Z_{k-1}\circ\phi_{k}-Z_{k-1}\label{eq:h1}\\
 & \quad+h_{0}\circ\phi_{k}-h_{0}-\left\{ g_{k},h_{0}\right\}_{\mathcal{S}} \label{eq:h2}\\
 & \quad+p_{k-1}\circ\phi_{k}-p_{k-1}\label{eq:h3}\\
 & \quad+\left(R_{k-1}-p_{k-1}\right)\circ\phi_{k}.\label{eq:h4}
\end{align}
In a classic scenario one has to solve the homological equation in
$g_{k}$ and $z_{k}$
\[
\left\{ g_{k},h_{0}\right\}_{\mathcal{S}} +p_{k-1}=z_{k}
\]
where $\left\{ z_{k},h_{0}\right\} _{\mathcal{S}}=0$ so that the
$k$-th order terms cancel out and prove that the remainders \eqref{eq:h1}, \eqref{eq:h2}, \eqref{eq:h3}, and \eqref{eq:h4}
are all of order at least $k+1$. However since $\left\{ \cdot,h_{0}\right\} _{\mathcal{S}}$
is no longer homogeneous, the classic homological equation does not
have a solution in the space of polynomials. Fortunately, we can allow
\eqref{eq:linear_eqn} to output a higher order part $r_{k}$ such
that $g_{k}$ is still a polynomial. Then homological equation is
modified as the following
\begin{equation}
\begin{cases}
\left\{ g_{k},h_{0}\right\} _{\mathcal{S}}+p_{k-1}=z_{k}+r_{k},\\
\left\{ z_{k},h_{0}\right\} _{\mathcal{S}}=0,\\
\mathrm{ord}\left(r_{k}\right)\ge k+1.
\end{cases}\label{eq:homological_eqn}
\end{equation}
According to \eqref{eq:h0_action}, if $\left(\boldsymbol{K},\boldsymbol{L}\right)\in\mathscr{R}_{m-1}\setminus\mathscr{R}_{m},1\le m\le N-1$,
one can apply
\begin{align}
 & \left\{ -\frac{4\pi i}{A_{m+1}\left(\delta_{m}\left(\boldsymbol{K}\right)-\delta_{m}\left(\boldsymbol{L}\right)\right)}\beta^{\boldsymbol{K}+\sum^{N-1}_{j=m}\boldsymbol{e}_{j}}\bar{\beta}^{\boldsymbol{L}+\sum^{N-1}_{j=m}\boldsymbol{e}_{j}},h_{0}\right\} _{S}\nonumber \\
= & -\beta^{\boldsymbol{K}}\bar{\beta}^{\boldsymbol{L}}-\sum^{N-1}_{l=m+1}\frac{A_{l+1}}{A_{m+1}}\frac{\delta_{l}\left(\boldsymbol{K}\right)-\delta_{l}\left(\boldsymbol{L}\right)}{\delta_{m}\left(\boldsymbol{K}\right)-\delta_{m}\left(\boldsymbol{L}\right)}\boldsymbol{\beta}^{\boldsymbol{K}+\sum^{l-1}_{j=m}\boldsymbol{e}_{j}}\bar{\boldsymbol{\beta}}^{\boldsymbol{L}+\sum^{l-1}_{j=m}\boldsymbol{e}_{j}}\label{eq:partial_sol_homological_1}\\
 & -\frac{4\pi\omega}{A_{m+1}A_{N}}\frac{\delta_{N}\left(\boldsymbol{K}\right)-\delta_{N}\left(\boldsymbol{L}\right)}{\delta_{m}\left(\boldsymbol{K}\right)-\delta_{m}\left(\boldsymbol{L}\right)}\beta^{\boldsymbol{K}+\sum^{N-1}_{j=m}\boldsymbol{e}_{j}}\bar{\beta}^{\boldsymbol{L}+\sum^{N-1}_{j=m}\boldsymbol{e}_{j}}\nonumber 
\end{align}
to cancel $\beta^{\boldsymbol{K}}\bar{\beta}^{\boldsymbol{L}}$ in
$p_{k-1}$. When $\left(\boldsymbol{K},\boldsymbol{L}\right)\in\mathscr{R}_{N-1}\setminus\mathscr{R}_{N}$
we apply 
\begin{equation}
\left\{ -\frac{iA_{N}}{\omega\left(\delta_{N}\left(\boldsymbol{K}\right)-\delta_{N}\left(\boldsymbol{L}\right)\right)}\beta^{\boldsymbol{K}}\bar{\beta}^{\boldsymbol{L}},h_{0}\right\} _{S}=-\beta^{\boldsymbol{K}}\bar{\beta}^{\boldsymbol{L}}.\label{eq:partial_sol_homological_2}
\end{equation}
Thus we obtain the following proposition to solve the modified homological
equation \eqref{eq:homological_eqn}.
\begin{prop}
\label{prop:full_sol_homological_eqn}Given a  $k$-th order homogenous polynomial
\[
p_{k-1}=\sum_{\left|\boldsymbol{K}\right|+\left|\boldsymbol{L}\right|=k}c_{\boldsymbol{K},\boldsymbol{L}}\boldsymbol{\beta}^{\boldsymbol{K}}\bar{\boldsymbol{\beta}}^{\boldsymbol{L}}
\]
with all exponents pair $\left(\boldsymbol{K},\boldsymbol{L}\right)\in\mathscr{R}_{0}$,
then one solution to the modified homological equation \eqref{eq:homological_eqn}
is 
\begin{align*}
g_{k} & =\sum^{N}_{m=1}c_{\boldsymbol{K},\boldsymbol{L}}g_{k,m},\\
r_{k} & =\sum^{N}_{m=1}c_{\boldsymbol{K},\boldsymbol{L}}r_{k,m},\\
z_{k} & =\sum_{\left(\boldsymbol{K},\boldsymbol{L}\right)\in\mathscr{R}_{N}}c_{\boldsymbol{K},\boldsymbol{L}}\boldsymbol{\beta}^{\boldsymbol{K}}\bar{\boldsymbol{\beta}}^{\boldsymbol{L}},
\end{align*}
where {\small
\begin{align*}
g_{k,m} & =\sum_{\left(\boldsymbol{K},\boldsymbol{L}\right)\in\mathscr{R}_{m-1}\setminus\mathscr{R}_{m}}-\frac{4\pi i\boldsymbol{\beta}^{\boldsymbol{K}+\boldsymbol{E}_{m}}\bar{\boldsymbol{\beta}}^{\boldsymbol{L}+\boldsymbol{E}_{m}}}{A_{m+1}\left(\delta_{m}\left(\boldsymbol{K}\right)-\delta_{m}\left(\boldsymbol{L}\right)\right)},\quad1\le m\le N-1,\\
g_{k,N} & =\sum_{\left(\boldsymbol{K},\boldsymbol{L}\right)\in\mathscr{R}_{N-1}\setminus\mathscr{R}_{N}}-\frac{iA_{N}\boldsymbol{\beta}^{\boldsymbol{K}}\bar{\boldsymbol{\beta}}^{\boldsymbol{L}}}{\omega\left(\delta_{N}\left(\boldsymbol{K}\right)-\delta_{N}\left(\boldsymbol{L}\right)\right)},\\
r_{k,m} & =\sum_{\left(\boldsymbol{K},\boldsymbol{L}\right)\in\mathscr{R}_{m-1}\setminus\mathscr{R}_{m}}\sum^{N-1}_{l=m+1}-\frac{A_{l+1}}{A_{m+1}}\frac{\delta_{l}\left(\boldsymbol{K}\right)-\delta_{l}\left(\boldsymbol{L}\right)}{\delta_{m}\left(\boldsymbol{K}\right)-\delta_{m}\left(\boldsymbol{L}\right)}\boldsymbol{\beta}^{\boldsymbol{K}+\boldsymbol{E}_{m}-\boldsymbol{E}_{l}}\bar{\boldsymbol{\beta}}^{\boldsymbol{L}+\boldsymbol{E}_{m}-\boldsymbol{E}_{l}}\\
 & \qquad-\frac{4\pi\omega}{A_{m+1}A_{N}}\frac{\delta_{N}\left(\boldsymbol{K}\right)-\delta_{N}\left(\boldsymbol{L}\right)}{\delta_{m}\left(\boldsymbol{K}\right)-\delta_{m}\left(\boldsymbol{L}\right)}\boldsymbol{\beta}^{\boldsymbol{K}+\boldsymbol{E}_{m}}\bar{\boldsymbol{\beta}}^{\boldsymbol{L}+\boldsymbol{E}_{m}},\quad1\le m\le N-1,\\
r_{k,N} & =0.
\end{align*}
}{\small\par}
\end{prop}
One expects that at the $k$-th step
\begin{align*}
H_{k} & =H_{k-1}\circ\phi_{k}\\
 & =h_{0}+Z_{k-1}+R_{k-1}+\left\{ g_{k},h_{0}\right\} _{\mathcal{S}}\\
 & \quad+\sum_{l\ge1}\frac{1}{\left(l+1\right)!}\left\{ g_{k},\left\{ g_{k},\cdots\left\{ g_{k},h_{0}\right\} _{\mathcal{S}}\right\} _{\mathcal{S}}\right\} _{\mathcal{S}}\\
 & \quad+\sum_{l\ge1}\frac{1}{l!}\left\{ g_{k},\left\{ g_{k},\cdots\left\{ g_{k},Z_{k-1}\right\} _{\mathcal{S}}\right\} _{\mathcal{S}}\right\} _{\mathcal{S}}\\
 & \quad+\sum_{l\ge1}\frac{1}{l!}\left\{ g_{k},\left\{ g_{k},\cdots\left\{ g_{k},R_{k-1}\right\} _{\mathcal{S}}\right\} _{\mathcal{S}}\right\} _{\mathcal{S}}.
\end{align*}
Let $Z_k = Z_{k-1}+z_k$. Since $Z_{k-1}+R_{k-1}+\left\{ g_{k},h_{0}\right\} _{\mathcal{S}}=Z_{k}+R_{k-1}-p_{k-1}+r_{k}$,
we have 
\begin{align}
R_{k} & =H_{k}-h_{0}-Z_{k-1}-z_{k}\nonumber \\
 & =R_{k-1}-p_{k-1}+r_{k}\label{eq:Rk_1}\\
 & \quad+\sum_{l\ge1}\frac{1}{l!}\left\{ g_{k},\left\{ g_{k},\cdots\left\{ g_{k},Z_{k-1}\right\} _{\mathcal{S}}\right\} _{\mathcal{S}}\right\} _{\mathcal{S}}\label{eq:Rk_2}\\
 & \quad+\sum_{l\ge1}\frac{1}{\left(l+1\right)!}\left\{ g_{k},\left\{ g_{k},\cdots\left\{ g_{k},z_{k}-p_{k-1}+r_{k}\right\} _{\mathcal{S}}\right\} _{\mathcal{S}}\right\} _{\mathcal{S}}\label{eq:Rk_3}\\
 & \quad+\sum_{l\ge1}\frac{1}{l!}\left\{ g_{k},\left\{ g_{k},\cdots\left\{ g_{k},R_{k-1}\right\} _{\mathcal{S}}\right\} _{\mathcal{S}}\right\} _{\mathcal{S}}.\label{eq:Rk_4}
\end{align}
One has $\mathrm{ord}\left(R_{k-1}-p_{k-1}\right) = k+1, \mathrm{ord}\left(r_{k}\right) \ge k+2$, so that
\[
\mathrm{ord}\left(R_{k-1}-p_{k-1}+r_{k}\right) = k+1.
\]
Noting that by induction assumption $\mathscr{E}\left(Z_{k-1}\right)\subset\mathscr{C}$,
so that according to Proposition \ref{prop:full_sol_homological_eqn},
all the exponents of $z_{k},g_{k},r_{k}$ belong to $\mathscr{C}$.
Hence it follows that
\[
\mathscr{E}\left(Z_{k}\right) = \mathscr{E}\left(Z_{k-1}+z_{k}\right)\subset\mathscr{C}
\]
and for \eqref{eq:Rk_1}
\[
\mathscr{E}\left(R_{k-1}-p_{k-1}+r_{k}\right)\subset\mathscr{C}.
\]
By the previous calculation \eqref{eq:poly_action} and noting that
$\delta_{l}\left(\boldsymbol{K}+\boldsymbol{E}_{m}\right)-\delta_{l}\left(\boldsymbol{L}+\boldsymbol{E}_{m}\right)=0$
for $l\le m-1,\left(\boldsymbol{K},\boldsymbol{L}\right)\in\mathscr{R}_{m-1}\setminus\mathscr{R}_{m}$,
we have for $\left(\boldsymbol{K},\boldsymbol{L}\right)\in\mathscr{R}_{m-1}\setminus\mathscr{R}_{m}$
\begin{align*}
 & \left\{ \beta^{\boldsymbol{K}+\boldsymbol{E}_{m}}\bar{\beta}^{\boldsymbol{L}+\boldsymbol{E}_{m}},\boldsymbol{\beta}^{\boldsymbol{M}}\bar{\boldsymbol{\beta}}^{\boldsymbol{M}}\right\} _{\mathcal{S}}\\
= & \sum^{N-1}_{l=m}c_{l}\left(\delta_{l}\left(\boldsymbol{K}\right)-\delta_{l}\left(\boldsymbol{L}\right)\right)\delta_{l}\left(\boldsymbol{M}\right)\boldsymbol{\beta}^{\boldsymbol{K}+\boldsymbol{M}+\boldsymbol{E}_{m}-\boldsymbol{E}_{l}}\bar{\boldsymbol{\beta}}^{\boldsymbol{L}+\boldsymbol{M}+\boldsymbol{E}_{m}-\boldsymbol{E}_{l}}\\
 & +c_{N}\left(\delta_{N}\left(\boldsymbol{K}\right)-\delta_{N}\left(\boldsymbol{L}\right)\right)\delta_{N}\left(\boldsymbol{M}\right)\boldsymbol{\beta}^{\boldsymbol{K}+\boldsymbol{M}+\boldsymbol{E}_{m}-\boldsymbol{e}_{N}}\bar{\boldsymbol{\beta}}^{\boldsymbol{L}+\boldsymbol{M}+\boldsymbol{E}_{m}-\boldsymbol{e}_{N}},
\end{align*}
where we apply the fact 
\[
\delta_{l}\left(\boldsymbol{K}+\boldsymbol{E}_{m}\right)-\delta_{l}\left(\boldsymbol{L}+\boldsymbol{E}_{m}\right)=\delta_{l}\left(\boldsymbol{K}\right)-\delta_{l}\left(\boldsymbol{L}\right).
\]
Then the possible minimum degree terms of $\left\{ g_{l,m},\boldsymbol{\beta}^{\boldsymbol{M}}\bar{\boldsymbol{\beta}}^{\boldsymbol{M}}\right\} _{\mathcal{S}}$
are 
\begin{align*}
c_{m}\left(\delta_{m}\left(\boldsymbol{K}\right)-\delta_{m}\left(\boldsymbol{L}\right)\right)\delta_{m}\left(\boldsymbol{M}\right)\boldsymbol{\beta}^{\boldsymbol{K}+\boldsymbol{M}}\bar{\boldsymbol{\beta}}^{\boldsymbol{L}+\boldsymbol{M}},\quad & \text{when }m\le N-1,\\
c_{N}\left(\delta_{N}\left(\boldsymbol{K}\right)-\delta_{N}\left(\boldsymbol{L}\right)\right)\delta_{N}\left(\boldsymbol{M}\right)\boldsymbol{\beta}^{\boldsymbol{K}+\boldsymbol{M}-\boldsymbol{e}_{N}}\bar{\boldsymbol{\beta}}^{\boldsymbol{L}+\boldsymbol{M}-\boldsymbol{e}_{N}},\quad & \text{when }m=N.
\end{align*}
Since $\boldsymbol{M}\neq\boldsymbol{e}_{N}$ if $\left|\boldsymbol{M}\right|=1$,
one must has $\delta_{N}\left(\boldsymbol{M}\right)=M_{N}=0$. Therefore
the case $m=N$ survives only when $\left|\boldsymbol{M}\right|\ge2$.
So that in both cases the degree is greater or equal to $k+2$, i.e.
\[
\mathrm{ord}\left(\left\{ g_{k},Z_{k-1}\right\} _{\mathcal{S}}\right)\ge k+2.
\]
Then $\mathrm{ord}\left(\eqref{eq:Rk_2}\right) \ge k + 2$ is true by applying the following Lemma (the proof is in Appendix \ref{sec:proof_oder_increase}) to \eqref{eq:Rk_2}
\begin{lem}
\label{claim:oder_increase}If $\mathscr{E}\left(R_{k-1}\right)\subset\mathscr{C}$
and $g_{k}$ is obtained via Proposition \ref{prop:full_sol_homological_eqn},
then for any polynomial $f$, it holds that 
\[
\mathrm{ord}\left(\left\{ g_{k},f\right\} \right)\ge\mathrm{ord}\left(f\right)+1.
\]
\end{lem}
To prove that $\mathscr{E}\left(\eqref{eq:Rk_2}\right)\subset\mathscr{C}$, we just need the Lemma below (see its proof in Appendix \ref{sec:proof_remainder_structure})
\begin{lem}
\label{claim:remainder_structure} If $\left(\boldsymbol{K},\boldsymbol{L}\right)\in\mathscr{C}\cap\left(\mathscr{R}_{m-1}\setminus\mathscr{R}_{m}\right)$
for some $1\le m\le N$, then for any $\left(\boldsymbol{K}',\boldsymbol{L}'\right)\in\mathscr{C}$
it holds that 
\[
\mathscr{E}\left(\left\{ \boldsymbol{\beta}^{\boldsymbol{K}+\boldsymbol{E}_{m}}\bar{\boldsymbol{\beta}}^{\boldsymbol{L}+\boldsymbol{E}_{m}},\boldsymbol{\beta}^{\boldsymbol{K}'}\bar{\boldsymbol{\beta}}^{\boldsymbol{L}'}\right\} _{\mathcal{S}}\right)\subset\mathscr{C}.
\]
\end{lem}
For \eqref{eq:Rk_3} and \eqref{eq:Rk_4}, by induction assumption and Proposition \ref{prop:full_sol_homological_eqn}, we already know that 
\[
\mathscr{E}\left(z_k-p_{k-1}+r_k\right)\subset\mathscr{C},\quad \mathscr{E}\left(R_{k-1}\right)\subset\mathscr{C}.
\]
Thus by applying Lemma \ref{claim:oder_increase} and Lemma \ref{claim:remainder_structure}, it holds that 
\[
\mathrm{ord}\left(\eqref{eq:Rk_3}, \eqref{eq:Rk_4}\right) \ge k+1, \quad \mathscr{E}\left(\eqref{eq:Rk_3}, \eqref{eq:Rk_4}\right)\subset\mathscr{C}.
\]
Then applying Lemma \ref{claim:oder_increase} to \eqref{eq:Rk_3} and \eqref{eq:Rk_4} yields
\[
\mathrm{ord}\left(R_{k}\right)\ge k+1.
\]
Summarizing the above discussion, the formal normal form at $k$-th
step is valid and 
\[
\mathscr{E}\left(R_{k}\right)\subset\mathscr{C},\mathscr{E}\left(Z_{k}\right)\subset\mathscr{C}.
\]
So that the normal form construction can proceed by applying Lemma
\ref{claim:oder_increase} and \ref{claim:remainder_structure} repeatedly.
Then Lemma \ref{lem:normal_form_formal} is proved by defining recursively
\[
T_{n}=T_{n-1}\circ\phi_{n},\cdots,T_{1}=\phi_{1}.
\]
\end{proof}

\section{A rigorous Birkhoff theorem}

In this section, we show the domain where the Lie transform in Lemma
\ref{lem:normal_form_formal} is well-defined as well as the corresponding
deformation estimates. Most importantly, we will state the rigorous
Birkhoff type theorem. Given $\delta_{1},\delta_{2}>0$, define the
space 
\[
U_{\delta_{1},\delta_{2}}:=\left(\bigcap^{N-1}_{k=1}\left\{ \delta^{N-k}_{2}\le\left|\alpha_{k}\right|,\left|\bar{\alpha}_{k}\right|\le\delta^{N-k}_{1}\right\} \right)\cap\left\{ \left|\alpha_{N}\right|,\left|\bar{\alpha}_{N}\right|\le\delta_{1}\right\} .
\]
First, let us define the domain where the Lie transform is well-defined.
\begin{lem}
\label{lem:g_domain} Assuming that the $k-1$-th normal form step
\[
H_{k-1}=h_{0}+Z_{k-1}+R_{k-1}
\]
holds, then for $g_{k}$ obtained by solving the modified homological
equation \eqref{eq:homological_eqn}, there exist a small enough $\varepsilon_{k}$
and positive constant $C_{k}$ such that for all $0<\varepsilon\le\varepsilon_{k}$
and $0\le\lambda\le1$, $g_{k}$ as a function of $\boldsymbol{\alpha},\bar{\boldsymbol{\alpha}}$
defines a Lie transform $\phi_{k}:U_{\varepsilon,\frac{\varepsilon}{2+\lambda}}\to U_{2\varepsilon,\frac{\varepsilon}{4+\lambda}}$
and the following deformation estimates hold
\[
\begin{cases}
\left|\left(\phi_{k}\left(\boldsymbol{\alpha},\bar{\boldsymbol{\alpha}}\right)\right)_{l}-\alpha_{l}\right|\le C_{k}\varepsilon^{k+N+l-2M\left(k,l\right)}, & 1\le l\le N-1,\\
\left|\left(\phi_{k}\left(\boldsymbol{\alpha},\bar{\boldsymbol{\alpha}}\right)\right)_{N}-\alpha_{N}\right|\le C_{k}\varepsilon^{k-1}, & l=N.
\end{cases}
\]
and 
\[
\begin{cases}
\left|\left(\overline{\phi_{k}\left(\boldsymbol{\alpha},\bar{\boldsymbol{\alpha}}\right)}\right)_{l}-\bar{\alpha}_{l}\right|\le C_{k}\varepsilon^{k+N+l-2M\left(k,l\right)}, & 1\le l\le N-1,\\
\left|\left(\overline{\phi_{k}\left(\boldsymbol{\alpha},\bar{\boldsymbol{\alpha}}\right)}\right)_{N}-\bar{\alpha}_{N}\right|\le C_{k}\varepsilon^{k-1}, & l=N.
\end{cases}
\]
where $M\left(k,l\right)=\min\left(N,l+\left\lfloor \frac{k-1}{2}\right\rfloor \right)$. 
\end{lem}
\begin{proof}
For $\left(\boldsymbol{K},\boldsymbol{L}\right)\in\mathscr{R}_{m-1}\setminus\mathscr{R}_{m}$,
rewriting $f_{\boldsymbol{K},\boldsymbol{L}}=\boldsymbol{\beta}^{\boldsymbol{K}+\boldsymbol{E}_m}\bar{\boldsymbol{\beta}}^{\boldsymbol{L}+\boldsymbol{E}_m}$
in terms of $\boldsymbol{\alpha},\bar{\boldsymbol{\alpha}}$ gives
\[
f_{\boldsymbol{K},\boldsymbol{L}}=\begin{cases}
\boldsymbol{\alpha}^{\boldsymbol{K}-\widehat{\boldsymbol{K}}+\boldsymbol{e}_{m}}\bar{\boldsymbol{\alpha}}^{\boldsymbol{L}-\widehat{\boldsymbol{L}}+\boldsymbol{e}_{m}}, & 1\le m\le N-1,\\
\boldsymbol{\alpha}^{\boldsymbol{K}-\widehat{\boldsymbol{K}}}\bar{\boldsymbol{\alpha}}^{\boldsymbol{L}-\widehat{\boldsymbol{L}}}, & m=N,
\end{cases}
\]
where
\[
\widehat{\boldsymbol{K}}=\left(0,K_{1},K_{2},\cdots,K_{N-2},0\right).
\]
Taking derivative w.r.t $\bar{\alpha}_{l}$ and rewriting it in $\left(\boldsymbol{\beta},\bar{\boldsymbol{\beta}}\right)$
yields for some $1\le m\le N$, 
\[
\partial_{\bar{\alpha}_{l}}f_{\boldsymbol{K},\boldsymbol{L}}=\begin{cases}
c_{l,\boldsymbol{L}}\boldsymbol{\beta}^{\boldsymbol{K}+\boldsymbol{E}_{m}}\bar{\boldsymbol{\beta}}^{\boldsymbol{L}+\boldsymbol{E}_{m}-\boldsymbol{E}_{l}}, & i_{\boldsymbol{L},m}\le l\le N-1,\\
c_{N,\boldsymbol{L}}\boldsymbol{\beta}^{\boldsymbol{K}+\boldsymbol{E}_{m}}\bar{\boldsymbol{\beta}}^{\boldsymbol{L}+\boldsymbol{E}_{m}-\boldsymbol{e}_{N}}, & l=N.
\end{cases}
\]
Indeed $\partial_{\bar{\alpha}_{l}}f_{\boldsymbol{K},\boldsymbol{L}}$
is the $l$-th component of the Hamiltonian vector field generated
by $f_{\boldsymbol{K},\boldsymbol{L}}$. Then the $l$-th component
of $X_{g_{k}}$ is for $1\le l\le N-1,$
\begin{equation}
\left(X_{g_{k}}\right)_{l}=\sum^{N}_{m=1}\sum_{\left(\boldsymbol{K},\boldsymbol{L}\right)\in\mathscr{R}_{m-1}\setminus\mathscr{R}_{m}}c_{\boldsymbol{K},\boldsymbol{L}}c_{l,\boldsymbol{L}}\boldsymbol{\beta}^{\boldsymbol{K}+\boldsymbol{E}_{m}}\bar{\boldsymbol{\beta}}^{\boldsymbol{L}+\boldsymbol{E}_{m}-\boldsymbol{E}_{l}}\theta\left(l-i_{\boldsymbol{L},m}\right),\label{eq:Ham_vector_field_1}
\end{equation}
and for $l=N$,
\begin{equation}
\left(X_{g_{k}}\right)_{N}=\sum^{N}_{m=1}\sum_{\left(\boldsymbol{K},\boldsymbol{L}\right)\in\mathscr{R}_{m-1}\setminus\mathscr{R}_{m}}c_{\boldsymbol{K},\boldsymbol{L}}c_{l,\boldsymbol{L}}\boldsymbol{\beta}^{\boldsymbol{K}+\boldsymbol{E}_{m}}\bar{\boldsymbol{\beta}}^{\boldsymbol{L}+\boldsymbol{E}_{m}-\boldsymbol{e}_{N}},\label{eq:Ham_vector_field_2}
\end{equation}
where 
\[
\theta\left(x\right)=\begin{cases}
1, & x\ge0,\\
0, & x<0,
\end{cases}
\]
and $i_{\boldsymbol{L},m}$ is the index of the first non-zero entry
of 
\[
\begin{cases}
\boldsymbol{L}-\widehat{\boldsymbol{L}}+\boldsymbol{e}_{m}, & 1\le m\le N-1,\\
\boldsymbol{L}-\widehat{\boldsymbol{L}}, & m=N.
\end{cases}
\]
Let 
\[
s^{*}=\inf_{\left(\boldsymbol{\alpha},\bar{\boldsymbol{\alpha}}\right)\in U_{\varepsilon,\frac{\varepsilon}{2+\lambda}}}\sup\left\{ s>0:\phi^{\tau}_{k}\left(\boldsymbol{\alpha},\bar{\boldsymbol{\alpha}}\right)\in U_{2\varepsilon,\frac{\varepsilon}{4+2\lambda}},\forall\tau\le s\right\} .
\]
Then with \eqref{eq:Ham_vector_field_1} and \eqref{eq:Ham_vector_field_2},
we can measure the size of $X_{g_{k}}\left(\phi^{s}_{k}\left(\boldsymbol{\alpha},\bar{\boldsymbol{\alpha}}\right)\right)$
for any $s\le s^{*},\left(\boldsymbol{\alpha},\bar{\boldsymbol{\alpha}}\right)\in U_{\varepsilon,\frac{\varepsilon}{2+\lambda}}$.
For $1\le l\le N-1$, we have $\left(\boldsymbol{K},\boldsymbol{L}\right)\in\left(\mathscr{R}_{m-1}\setminus\mathscr{R}_{m}\right)\cap\mathscr{C}$
by Lemma \ref{lem:normal_form_formal}. Applying Proposition \ref{prop:index_inequality}
gives 
\[
k-2m+2i_{\boldsymbol{L},m}\ge1.
\]
Therefore, $\theta\left(l-i_{\boldsymbol{L},m}\right)=1$ implies
$k-2m+2l\ge1$ and plugging it into \eqref{eq:Ham_vector_field_1}
yields that for $1\le l\le N-1$,
\begin{align*}
 & \left|\left(X_{g_{k}}\left(\phi^{s}_{k}\left(\boldsymbol{\alpha},\bar{\boldsymbol{\alpha}}\right)\right)\right)_{l}\right|\\
\le & C_{k}\sum^{N}_{m=1}\sum_{\left(\boldsymbol{K},\boldsymbol{L}\right)\in\mathscr{R}_{m-1}\setminus\mathscr{R}_{m}}\theta\left(l-i_{\boldsymbol{L},m}\right)\varepsilon^{k+2\left(N-m\right)-\left(N-l\right)}\\
\le & C_{k}\sum^{N}_{m=1}\sum_{k-2m+2l\ge1}\varepsilon^{k+N-2m+l}\\
= & C_{k}\frac{1-\varepsilon^{2M}}{1-\varepsilon^{2}}\varepsilon^{k+N+l-2M\left(k,l\right)},
\end{align*}
where $M\left(k,l\right)=\min\left(N,l+\left\lfloor \frac{k-1}{2}\right\rfloor \right)$.
For small $\varepsilon$, it holds 
\begin{equation}
\left|\left(X_{g_{k}}\left(\phi^{s}_{k}\left(\boldsymbol{\alpha},\bar{\boldsymbol{\alpha}}\right)\right)\right)_{l}\right|\le C_{k}\varepsilon^{k+N+l},\quad1\le l\le N-1.\label{eq:estimate_ham_vector_1}
\end{equation}
Similarly, for the $N$-th component of $X_{g_{k}}\left(\phi^{s}_{k}\left(\boldsymbol{\alpha},\bar{\boldsymbol{\alpha}}\right)\right)$,
we have 
\begin{equation}
\left|\left(X_{g_{k}}\left(\phi^{s}_{k}\left(\boldsymbol{\alpha},\bar{\boldsymbol{\alpha}}\right)\right)\right)_{N}\right|\le C_{k}\frac{1-\varepsilon^{2N}}{1-\varepsilon^{2}}\varepsilon^{k-1}\le C_{k}\varepsilon^{k-1}.\label{eq:estimate_ham_vector_2}
\end{equation}
Again by definition of $s^{*}$, there exist $\left(\boldsymbol{\alpha}^{out},\bar{\boldsymbol{\alpha}}^{out}\right),\left(\boldsymbol{\alpha}^{in},\bar{\boldsymbol{\alpha}}^{in}\right)\in U_{\varepsilon,\frac{\varepsilon}{2+\gamma}}$
such that for $1\le l\le N-1$
\begin{align*}
\left(\phi^{s^{*}}_{k}\left(\boldsymbol{\alpha}^{in},\bar{\boldsymbol{\alpha}}^{in}\right)\right)_{l} & =\left(\frac{\varepsilon}{4+2\lambda}\right)^{N-l},\\
\left(\phi^{s^{*}}_{k}\left(\boldsymbol{\alpha}^{out},\bar{\boldsymbol{\alpha}}^{out}\right)\right)_{l} & =\left(2\varepsilon\right)^{N-l},
\end{align*}
and 
\[
\left(\phi^{s^{*}}_{k}\left(\boldsymbol{\alpha}^{out},\bar{\boldsymbol{\alpha}}^{out}\right)\right)_{N}=2\varepsilon.
\]
With \eqref{eq:Ham_vector_field_1}, it follows that for $1\le l\le N-1$
\begin{align*}
\left(\frac{\varepsilon}{4+2\lambda}\right)^{N-l} & =\left|\left(\phi^{s^{*}}_{k}\left(\boldsymbol{\alpha}^{in},\bar{\boldsymbol{\alpha}}^{in}\right)\right)_{l}\right|\\
 & =\left|\alpha^{in}_{l}+\int^{s^{*}}_{0}\frac{d}{ds}\left(\phi^{s}_{k}\left(\boldsymbol{\alpha}^{in},\bar{\boldsymbol{\alpha}}^{in}\right)\right)_{l}ds\right|\\
 & =\left|\alpha^{in}_{l}+\int^{s^{*}}_{0}\left(X_{g_{k}}\left(\phi^{s}_{k}\left(\boldsymbol{\alpha}^{in},\bar{\boldsymbol{\alpha}}^{in}\right)\right)\right)_{l}ds\right|\\
 & \ge\left(\frac{\varepsilon}{2+\lambda}\right)^{N-l}-C_{k}\varepsilon^{k+N+l-2M\left(k,l\right)}s^{*}
\end{align*}
which implies 
\begin{equation}
s^{*}\ge\frac{1-1/2^{N-l}}{\left(2+\lambda\right)^{N-l}C_{k}\varepsilon^{k+2l-2M\left(k,l\right)}}\ge\frac{1}{2\cdot3^{N}C_{k}\varepsilon^{k+2l-2M\left(k,l\right)}}.\label{eq:s_lower_bd_1}
\end{equation}
Similarly for $\left(\boldsymbol{\alpha}^{out},\bar{\boldsymbol{\alpha}}^{out}\right)$,
one has for $1\le l\le N-1$
\begin{align*}
\left(2\varepsilon\right)^{N-l} & =\left|\left(\phi^{s^{*}}_{k}\left(\boldsymbol{\alpha}^{out},\bar{\boldsymbol{\alpha}}^{out}\right)\right)_{l}\right|\\
 & \le\left|\alpha^{out}_{l}\right|+\left|\int^{s^{*}}_{0}\left(X_{g_{k}}\left(\phi^{s}_{k}\left(\boldsymbol{\alpha}^{out},\bar{\boldsymbol{\alpha}}^{out}\right)\right)\right)_{l}ds\right|\\
 & \le\varepsilon^{N-l}+C_{k}\varepsilon^{k+N+l-2M\left(k,l\right)}s^{*},
\end{align*}
and for $l=N$, it holds that
\begin{align*}
2\varepsilon & =\left|\left(\phi^{s^{*}}_{k}\left(\boldsymbol{\alpha}^{out},\bar{\boldsymbol{\alpha}}^{out}\right)\right)_{N}\right|\\
 & \le\left|\alpha^{out}_{N}\right|+\left|\int^{s^{*}}_{0}\left(X_{g_{k}}\left(\phi^{s}_{k}\left(\boldsymbol{\alpha}^{out},\bar{\boldsymbol{\alpha}}^{out}\right)\right)\right)_{N}ds\right|\\
 & \le\varepsilon+C_{k}\varepsilon^{k-1}s^{*}.
\end{align*}
Then we obtain another two inequalities 
\begin{align}
s^{*} & \ge\frac{2^{N-l}-1}{C_{k}\varepsilon^{k+2l-2M\left(k,l\right)}}\ge\frac{1}{C_{k}\varepsilon^{k+2l-2M\left(k,l\right)}},\quad1\le l\le N-1,\label{eq:s_lower_bd_2}\\
s^{*} & \ge\frac{1}{C_{k}\varepsilon^{k-2}}.\label{eq:s_lower_bd_3}
\end{align}
For all $k\ge1$, we have 
\begin{equation}
k+2l-2M\left(k,l\right)\ge k-2\left\lfloor \frac{k-1}{2}\right\rfloor \ge1.\label{eq:index_estimate}
\end{equation}
Meanwhile according to the structure of $R_{0}$ for $k=1,2$, 
\[
\boldsymbol{K}_{N}=\boldsymbol{L}_{N}=0,
\]
so $\phi_{g}$ does not deform $\alpha_{N},\bar{\alpha}_{N}$. Then
with \eqref{eq:s_lower_bd_1}, \eqref{eq:s_lower_bd_2}, and \eqref{eq:s_lower_bd_3},
there exists $\varepsilon_{k}$ small such that for all $0<\varepsilon\le\varepsilon_{k}$,
it holds that $s^{*}>1$, which means 
\[
\phi_{k}:U_{\varepsilon,\frac{\varepsilon}{2+\lambda}}\to U_{2\varepsilon,\frac{\varepsilon}{4+2\lambda}}
\]
is well-defined. 

For the deformation estimates, it is easy to see that for $1\le l\le N-1$,
we have 
\[
\left|\left(\phi_{k}\left(\boldsymbol{\alpha},\bar{\boldsymbol{\alpha}}\right)\right)_{l}-\alpha_{l}\right|\le\int^{1}_{0}\left|\left(X_{g_{k}}\left(\phi^{s}_{k}\left(\boldsymbol{\alpha},\bar{\boldsymbol{\alpha}}\right)\right)\right)_{l}\right|ds\le C_{k}\varepsilon^{k+N+l-2M\left(k,l\right)},
\]
and 
\[
\left|\left(\phi_{k}\left(\boldsymbol{\alpha},\bar{\boldsymbol{\alpha}}\right)\right)_{N}-\alpha_{N}\right|\le\int^{1}_{0}\left|\left(X_{g_{k}}\left(\phi^{s}_{k}\left(\boldsymbol{\alpha},\bar{\boldsymbol{\alpha}}\right)\right)\right)_{N}\right|ds\le C_{k}\varepsilon^{k-1}.
\]
The proof for the conjugate version is similar. 
\end{proof}

The rigorous Birkhoff type theorem is the following.
\begin{thm}
\label{thm:birkhoff}Given $n\in\mathbb{N}^{+}$, there exist a small
enough constant $\varepsilon_{0}$ and a constant $C_{n}$ such that
for all $0<\varepsilon\le\varepsilon_{0}$, there is a nearly identity
canonical transform $T_{n}:U_{\varepsilon,\frac{\varepsilon}{2}}\to U_{2\varepsilon,\frac{\varepsilon}{4}}$
which puts the Hamiltonian in the Birkhoff normal form up to order
$n$, namely 
\[
H_{n}:=H_\Omega\circ T_{n}=h_{0}+Z_{n}+R_{n},
\]
where $\left\{ Z_{n},H_{0}\right\} _{\mathcal{S}}=0$ and
\[
\left|R_{n}\right|\le C_{n}\varepsilon^{n+1}.
\]
Moreover one has the following deformation estimates: 
\begin{align*}
\left|\left(T_{n}\left(\boldsymbol{\alpha},\bar{\boldsymbol{\alpha}}\right)\right)_{l}-\alpha_{l}\right| & \le C_{n}\varepsilon^{N-l+1},\quad1\le l\le N-1,\\
\left|\left(T_{n}\left(\boldsymbol{\alpha},\bar{\boldsymbol{\alpha}}\right)\right)_{N}-\alpha_{N}\right| & \le C_{n}\varepsilon^{2}.
\end{align*}
\end{thm}
\begin{proof}
By Lemma \ref{lem:normal_form_formal}, we already know all the generating
function $g_{k},k=1,\cdots,n$. Now we construct the domain for each
$g_{k}$ in reverse order. Starting from $g_{n}$, by Lemma \ref{lem:g_domain},
there exists $\varepsilon_{n}$ small enough and a constant $C_{n}$
such that for all $0\le\sigma_{n}\le\varepsilon_{n}$, $\phi_{n}:U_{\sigma_{n},\frac{\sigma_{n}}{2+\lambda_{n}}}\to U_{2\sigma_{n},\frac{\sigma_{n}}{4+2\lambda_{n}}}$
with $\lambda_{n}=0$ is well-defined and 
\[
\begin{cases}
\left|\left(\phi_{n}\left(\boldsymbol{\alpha},\bar{\boldsymbol{\alpha}}\right)\right)_{l}-\alpha_{l}\right|\le C_{n}\sigma^{n+N+l-2M\left(n,l\right)}_{n}, & 1\le l\le N-1,\\
\left|\left(\phi_{n}\left(\boldsymbol{\alpha},\bar{\boldsymbol{\alpha}}\right)\right)_{N}-\alpha_{N}\right|\le C_{n}\sigma^{n-1}_{n}, & l=N.
\end{cases},
\]
where $M\left(n,l\right)=\min\left(N,l+\left\lfloor \frac{n-1}{2}\right\rfloor \right)$.
By the above deformation estimates, the image of $\phi_{n}$ is a
subset of 
\[
U_{\sigma_{n}\left(1+C_{n}\sigma_{n}\right),\frac{\sigma_{n}}{2+\lambda_{n}}\left(1-3^{N}C_{n}\sigma_{n}\right)}.
\]
Indeed, it holds that for $l=N$,
\[
\left|\left(\phi_{n}\left(\boldsymbol{\alpha},\bar{\boldsymbol{\alpha}}\right)\right)_{N}\right|\le\sigma_{n}+C_{n}\sigma^{n-1}_{n}\le\sigma_{n}\left(1+C_{n}\sigma_{n}\right),
\]
and for $1\le l\le N-1$, 
\begin{align*}
 & \left|\left(\phi_{n}\left(\boldsymbol{\alpha},\bar{\boldsymbol{\alpha}}\right)\right)_{l}\right|\\
\le & \sigma^{N-l}_{n}+C_{n}\sigma^{n+N+l-2M\left(n,l\right)}_{n}\\
= & \sigma^{N-l}_{n}\left(1+C_{n}\sigma^{n+2l-2M\left(n,l\right)}_{n}\right)\\
\le & \sigma^{N-l}_{n}\left(1+C_{n}\sigma_{n}\right)^{N-l},
\end{align*}
where we use the fact that $n+2l-2M\left(n,l\right)\ge1$ at the last
step. The following lower bound also holds for $1\le l\le N-1$
\begin{align*}
 & \left|\left(\phi_{n}\left(\boldsymbol{\alpha},\bar{\boldsymbol{\alpha}}\right)\right)_{l}\right|\\
\ge & \frac{\sigma^{N-l}_{n}}{\left(2+\lambda_{n}\right)^{N-l}}-C_{n}\sigma^{n+N+l-2M\left(n,l\right)}_{n}\\
\ge & \frac{\sigma^{N-l}_{n}}{\left(2+\lambda_{n}\right)^{N-l}}\left(1-\left(2+\lambda_{n}\right)^{N-l}C_{n}\sigma_{n}\right)\\
\ge & \frac{\sigma^{N-l}_{n}}{\left(2+\lambda_{n}\right)^{N-l}}\left(1-3^{N}C_{n}\sigma_{n}\right)^{N-l}.
\end{align*}
For small $\sigma_{n}$, letting $\sigma_{n-1}=\sigma_{n}\left(1+C_{n}\sigma_{n}\right)$
and 
\[
\lambda_{n-1}=\frac{\lambda_{n}\left(1+C_{n}\sigma_{n}\right)+2\left(1+3^{N}\right)C_{n}\sigma_{n}}{1-3^{N}C_{n}\sigma_{n}}
\]
gives 
\[
\phi_{n}\left(U_{\sigma_{n},\frac{\sigma_{n}}{2+\lambda_{n}}}\right)\subset U_{\sigma_{n}\left(1+C_{n}\sigma_{n}\right),\frac{\sigma_{n}}{2+\lambda_{n}}\left(1-3^{N}C_{n}\sigma_{n}\right)}=U_{\sigma_{n-1},\frac{\sigma_{n-1}}{2+\lambda_{n-1}}}.
\]
For $\phi_{n-1}$, applying Lemma \ref{lem:g_domain} and assuming
$\varepsilon_{n-1}$ the threshold, for small enough $\sigma_{n}$
we have $\sigma_{n-1}\le\varepsilon_{n-1}$. Thus $\phi_{n-1}$ is
well-defined on $U_{\sigma_{n-1},\frac{\sigma_{n-1}}{2+\lambda_{n-1}}}$
and consequently on $\phi_{n}\left(U_{\sigma_{n},\frac{\sigma_{n}}{2+\lambda_{n}}}\right)$.
We continue this procedure until it ends. Since $n$ is finite, one
obtains a small enough $\varepsilon_{0}$ such that for all $0<\varepsilon\le\varepsilon_{0}$,
the following mapping chain is well-defined, 
\[
U_{\sigma_{n},\frac{\sigma_{n}}{2+\lambda_{n}}}\overset{\phi_{n}}{\longrightarrow}U_{\sigma_{n-1},\frac{\sigma_{n-1}}{2+\lambda_{n-1}}}\overset{\phi_{n-1}}{\longrightarrow}\cdots\overset{\phi_{2}}{\longrightarrow}U_{\sigma_{1},\frac{\sigma_{1}}{2+\lambda_{1}}}\overset{\phi_{1}}{\longrightarrow}U_{\sigma_{0},\frac{\sigma_{0}}{2+\lambda_{0}}}
\]
where $\lambda_{n}=0,\sigma_{n}=\varepsilon$ and $\lambda_{k},\sigma_{k}$
are defined in the above iteration procedure. Overall, $T_{n}=\phi_{1}\circ\phi_{2}\circ\cdots\circ\phi_{n}$
is well-defined on $U_{\varepsilon,\frac{\varepsilon}{2}}$. We choose
$\sigma_{n}=\varepsilon$ so small that all $\sigma_{j},j=1,\cdots,n$
are of order $\varepsilon$ and $\lambda_{j}$ are close to $0$.
Therefore it holds that 
\[
U_{\sigma_{0},\frac{\sigma_{0}}{2+\lambda_{0}}}\subset U_{2\varepsilon,\frac{\varepsilon}{4}}
\]
and $T_{n}:U_{\varepsilon,\frac{\varepsilon}{2}}\to U_{2\varepsilon,\frac{\varepsilon}{4}}$. 

It remains to show the deformation estimates. Let 
\[
\left(\boldsymbol{\alpha}^{\left(k\right)},\bar{\boldsymbol{\alpha}}^{\left(k\right)}\right)=\phi_{k}\circ\phi_{k+1}\circ\cdots\circ\phi_{n}\left(\boldsymbol{\alpha},\bar{\boldsymbol{\alpha}}\right).
\]
Then $\alpha^{\left(k\right)}_{l}=\left(\phi_{k}\circ\phi_{k+1}\circ\cdots\circ\phi_{n}\left(\boldsymbol{\alpha},\bar{\boldsymbol{\alpha}}\right)\right)_{l}$
and for $1\le l\le N-1$, it holds that
\begin{align*}
\left|\left(T_{n}\left(\boldsymbol{\alpha},\bar{\boldsymbol{\alpha}}\right)\right)_{l}-\alpha_{l}\right| & =\left|\left(\phi_{1}\circ\phi_{2}\cdots\circ\phi_{n}\left(\boldsymbol{\alpha},\bar{\boldsymbol{\alpha}}\right)\right)_{l}-\alpha_{l}\right|\\
 & \le\left|\alpha_{l}-\left(\phi_{n}\left(\boldsymbol{\alpha},\bar{\boldsymbol{\alpha}}\right)\right)_{l}\right|+\left|\alpha^{\left(n\right)}_{l}-\left(\phi_{n-1}\left(\boldsymbol{\alpha}^{\left(n\right)},\bar{\boldsymbol{\alpha}}^{\left(n\right)}\right)\right)_{l}\right|\\
 & \quad+\left|\alpha^{\left(n-1\right)}_{l}-\left(\phi_{n-2}\left(\boldsymbol{\alpha}^{\left(n-1\right)},\bar{\boldsymbol{\alpha}}^{\left(n-1\right)}\right)\right)_{l}\right|+\cdots\\
 & \quad+\left|\alpha^{\left(2\right)}_{l}-\left(\phi_{1}\left(\boldsymbol{\alpha}^{\left(2\right)},\bar{\boldsymbol{\alpha}}^{\left(2\right)}\right)\right)_{l}\right|.
\end{align*}
Applying the deformation estimates in Lemma \ref{lem:g_domain} yields
for $1\le l\le N-1$,
\[
\left|\left(T_{n}\left(\boldsymbol{\alpha},\bar{\boldsymbol{\alpha}}\right)\right)_{l}-\alpha_{l}\right|\le\sum^{n}_{k=1}C_{k}\varepsilon^{k+N+l-2M\left(k,l\right)}\le C_{n}\varepsilon^{N-l+1}.
\]
For $l=N$, $\phi_{1},\phi_{2}$ are actually do not move the $N$-th
direction. So in a similarly way one has 
\[
\left|\left(T_{n}\left(\boldsymbol{\alpha},\bar{\boldsymbol{\alpha}}\right)\right)_{N}-\alpha_{N}\right|\le\sum^{n}_{k=3}C_{k}\varepsilon^{k-1}\le C_{n}\varepsilon^{2}.
\]
\end{proof}

\section{Proof of the main Theorem \ref{thm:main_result}}

The proof is standard. We do it for the sake of completeness. Define
the exist time $\tau\left(\boldsymbol{\alpha}_{0},r\right)$ by 
\[
\tau\left(\boldsymbol{\alpha}_{0},r\right):=\inf\left\{ t:\left(\boldsymbol{\alpha}\left(t\right),\bar{\boldsymbol{\alpha}}\left(t\right)\right)\notin U_{r,\frac{r}{2}},\left(\boldsymbol{\alpha}_{0},\bar{\boldsymbol{\alpha}}_{0}\right)\in U_{r,\frac{r}{2}}\right\} ,
\]
where $\boldsymbol{\alpha}\left(t\right)$ is the solution to the
original normalized Hamiltonian system with the Hamiltonian \eqref{eq:Ham}
and $\boldsymbol{\alpha}_{0}$ is its initial condition. We choose
$\varepsilon$ small enough so that for $t<\tau\left(\boldsymbol{\alpha}_{0},2\varepsilon\right)$
it holds that 
\[
\left(\boldsymbol{\alpha}\left(t\right),\bar{\boldsymbol{\alpha}}\left(t\right)\right)\in U_{2\varepsilon,\varepsilon}.
\]
According to the Birkhoff Theorem \ref{thm:birkhoff}, there is a
canonical transform $T_{n}$ such that $T_{n}\left(\boldsymbol{\alpha}',\bar{\boldsymbol{\alpha}}'\right)=\left(\alpha,\bar{\alpha}\right)$
and 
\[
H_{n}\left(\boldsymbol{\alpha}',\bar{\boldsymbol{\alpha}}'\right)=H_\Omega\circ T_{n}\left(\boldsymbol{\alpha}',\bar{\boldsymbol{\alpha}}'\right)=h_{0}\left(\boldsymbol{\alpha}',\bar{\boldsymbol{\alpha}}'\right)+Z_{n}\left(\boldsymbol{\alpha}',\bar{\boldsymbol{\alpha}}'\right)+R_{n}\left(\boldsymbol{\alpha}',\bar{\boldsymbol{\alpha}}'\right).
\]
Meanwhile applying the Birkhoff Theorem \ref{thm:birkhoff} to the
inverse canonical transform $T^{-1}_{n}$, we know that for $t<\tau\left(\boldsymbol{\alpha}_{0},\varepsilon\right)$,
$\left(\boldsymbol{\alpha}'\left(t\right),\bar{\boldsymbol{\alpha}}'\left(t\right)\right)\in U_{4\varepsilon,2\varepsilon}$
and 
\begin{align}
\left|\alpha'_{l}\left(t\right)-\alpha_{l}\left(t\right)\right| & \le C_{n}\varepsilon^{N-l+1},\quad1\le l\le N-1,\nonumber \\
\left|\alpha'_{N}\left(t\right)-\alpha_{N}\left(t\right)\right| & \le C_{n}\varepsilon^{2},\label{eq:alpha_deformation_estimate}
\end{align}
where $\boldsymbol{\alpha}'\left(t\right)$ is the solution to the
system with the Hamiltonian $H_{n}$. Recall that the $l$-th action
variable is simply $I_{l}\left(\boldsymbol{\alpha},\bar{\boldsymbol{\alpha}}\right)=\alpha_{l}\bar{\alpha}_{l}$.
We consider that for $t<\tau\left(\alpha_{0},2\varepsilon\right)$,
\begin{align}
\left|I_{l}\left(\boldsymbol{\alpha}\left(t\right)\right)-I_{l}\left(\boldsymbol{\alpha}\left(0\right)\right)\right| & \le\left|I_{l}\left(\boldsymbol{\alpha}\left(t\right)\right)-I_{l}\left(\boldsymbol{\alpha}'\left(t\right)\right)\right|+\left|I_{l}\left(\boldsymbol{\alpha}\left(0\right)\right)-I_{i}\left(\boldsymbol{\alpha}'\left(0\right)\right)\right|\nonumber \\
 & \quad+\left|I_{l}\left(\boldsymbol{\alpha}'\left(t\right)\right)-I_{l}\left(\boldsymbol{\alpha}'\left(0\right)\right)\right|,\qquad1\le l\le N.\label{eq:diff_I}
\end{align}
For $1\le l\le N-1$, with the help of \eqref{eq:alpha_deformation_estimate},
the first term in \eqref{eq:diff_I} is estimated as follows
\begin{align*}
 & \left|I_{l}\left(\boldsymbol{\alpha}\left(t\right)\right)-I_{l}\left(\boldsymbol{\alpha}'\left(t\right)\right)\right|\\
\le & \frac{1}{2}\left|\alpha_{l}\left(t\right)+\alpha'_{l}\left(t\right)\right|\left|\bar{\alpha}_{l}\left(t\right)-\bar{\alpha}'_{l}\left(t\right)\right|+\frac{1}{2}\left|\alpha_{l}\left(t\right)-\alpha'_{l}\left(t\right)\right|\left|\bar{\alpha}_{l}\left(t\right)+\bar{\alpha}'_{l}\left(t\right)\right|\\
\le & C_{n}\varepsilon^{N-l}\varepsilon^{N-l+1}=C_{n}\varepsilon^{2N-2l+1}.
\end{align*}
The second term in the right hand side of \eqref{eq:diff_I} is estimated
in the same way. When $l=N$, the first and the second term of \eqref{eq:diff_I}
are controlled in a similar way and one has 
\begin{align*}
\left|I_{N}\left(\boldsymbol{\alpha}\left(t\right)\right)-I_{N}\left(\boldsymbol{\alpha}'\left(t\right)\right)\right| & \le C_{n}\varepsilon^{3},\\
\left|I_{N}\left(\boldsymbol{\alpha}\left(0\right)\right)-I_{N}\left(\boldsymbol{\alpha}'\left(0\right)\right)\right| & \le C_{n}\varepsilon^{3}.
\end{align*}
 For the third term, since $\boldsymbol{\alpha}'\left(t\right),\boldsymbol{\alpha}'\left(0\right)$
are all in $U_{2\varepsilon,\varepsilon}$, one has
\begin{align}
\left|I_{l}\left(\boldsymbol{\alpha}'\left(t\right)\right)-I_{l}\left(\boldsymbol{\alpha}'\left(0\right)\right)\right| & =\left|\int^{t}_{0}\left\{ H_{n},I_{l}\right\} _{\mathcal{S}}\left(\boldsymbol{\alpha}'\left(\tau\right)\right)d\tau\right|\nonumber \\
 & =\left|\int^{t}_{0}\left\{ R_{n},I_{I}\right\} _{\mathcal{S}}\left(\boldsymbol{\alpha}'\left(\tau\right)\right)d\tau\right|\nonumber \\
 & \le Ct\left|\left\{ R_{n},I_{l}\right\} _{\mathcal{S}}\right|.\label{eq:deformation_I}
\end{align}
Since $\left(\boldsymbol{\alpha}'\left(t\right),\bar{\boldsymbol{\alpha}}'\left(t\right)\right)\in U_{4\varepsilon,2\varepsilon}$,
so that $\left|\beta'_{l}\left(t\right)\right|\le C\varepsilon$ for
all $1\le l\le N$ and we can estimate \eqref{eq:deformation_I} for
each $l=1,2,\cdots,N$:
\[
\left|\left\{ R_{n},I_{l}\right\} _{\mathcal{S}}\left(\boldsymbol{\alpha}'\left(t\right)\right)\right|=\left|\sum_{\left|\boldsymbol{K}\right|+\left|\boldsymbol{L}\right|=n+1}C_{\boldsymbol{K},\boldsymbol{L}}\beta'^{\boldsymbol{K}}\bar{\beta}'^{\boldsymbol{L}}+h.o.t\right|\le C_{n}\varepsilon^{n+1}.
\]
Therefore, for $t\le\frac{1}{\varepsilon^{n-2\left(N-1\right)}}$,
\eqref{eq:diff_I} is controlled by 
\begin{align*}
\left|I_{l}\left(\boldsymbol{\alpha}\left(t\right)\right)-I_{l}\left(\boldsymbol{\alpha}\left(0\right)\right)\right| & \le C_{n}\varepsilon^{2N-2l+1}+C_{n}t\varepsilon^{n+1}\le C_{n}\varepsilon^{2N-2l+1},1\le l\le N-1.\\
\left|I_{N}\left(\boldsymbol{\alpha}\left(t\right)\right)-I_{N}\left(\boldsymbol{\alpha}\left(0\right)\right)\right| & \le C_{n}\varepsilon^{3}+C_{n}t\varepsilon^{n+1}\le C_{n}\varepsilon^{3}.
\end{align*}
We have proved that for $t\le\min\left\{ \frac{1}{\varepsilon^{n-4}},\tau\left(\boldsymbol{\alpha}_{0},2\varepsilon\right)\right\} $,
\begin{equation}
\begin{cases}
\left|I_{l}\left(\boldsymbol{\alpha}\left(t\right)\right)-I_{l}\left(\boldsymbol{\alpha}\left(0\right)\right)\right|\le C_{n}\varepsilon^{2N-2l+1}, & 1\le l\le N-1,\\
\left|I_{N}\left(\boldsymbol{\alpha}\left(t\right)\right)-I_{N}\left(\boldsymbol{\alpha}\left(0\right)\right)\right|\le C_{n}\varepsilon^{3}.
\end{cases}\label{eq:l_estimate}
\end{equation}
Next we prove that $\tau\left(\boldsymbol{\alpha}_{0},2\varepsilon\right)>\frac{1}{\varepsilon^{n-2\left(N-1\right)}}$.
Assume by contradiction that $\tau\left(\boldsymbol{\alpha}_{0},2\varepsilon\right)\le\frac{1}{\varepsilon^{n-2\left(N-1\right)}}$.
Then for $t=\tau\left(\boldsymbol{\alpha}_{0},2\varepsilon\right)$,
by definition of $U_{2\varepsilon,\varepsilon}$ we have
\begin{align}
\left|\boldsymbol{\alpha}\left(t\right)\right|^{2} & =\left(2\varepsilon\right)^{2\left(N-1\right)}+\left(2\varepsilon\right)^{2\left(N-2\right)}+\cdots+\left(2\varepsilon\right)^{2\left(N-\left(N-1\right)\right)}+\left(2\varepsilon\right)^{2}\nonumber \\
 & =8\varepsilon^{2}+16\varepsilon^{4}+64\varepsilon^{6}+\cdots4^{N-1}\varepsilon^{2\left(N-1\right)}.\label{eq:contridiction_1}
\end{align}
On the other hand, one has 
\begin{align}
\left|\boldsymbol{\alpha}\left(t\right)\right|^{2} & =\sum^{N}_{l=1}I_{l}\left(\boldsymbol{\alpha}\left(t\right)\right)\nonumber \\
 & \le\sum^{N}_{l=1}I_{l}\left(\boldsymbol{\alpha}\left(0\right)\right)+\left|I_{l}\left(\boldsymbol{\alpha}\left(t\right)\right)-I_{l}\left(\boldsymbol{\alpha}\left(0\right)\right)\right|\nonumber \\
 & \le2\varepsilon^{2}+\varepsilon^{4}+\varepsilon^{6}+\cdots+\varepsilon^{2\left(N-1\right)}+NC_{n}\varepsilon^{3},\label{eq:contridiction_2}
\end{align}
where the last inequality holds because of the assumption $\tau>\frac{1}{\varepsilon^{n-2\left(N-1\right)}}$
and \eqref{eq:l_estimate}. Combining \eqref{eq:contridiction_1} and \eqref{eq:contridiction_2}
yields that for small $\varepsilon$, 
\[
6\varepsilon^{2}+15\varepsilon^{4}+63\varepsilon^{6}+\cdots+\left(4^{N-1}-1\right)\varepsilon^{2\left(N-1\right)}\le NC_{n}\varepsilon^{3},
\]
which is absurd. 

\section{Applications}\label{sec:app}

\subsection{Hierarchical clustering configurations}

We can say that the configuration introduced by the main Theorem \ref{thm:main_result}
is one cluster with hierarchical structure. Actually our method can
apply to the two-based hierarchical clustering configuration (see
Fig \ref{fig:HierNested}). 
\begin{figure}
\begin{centering}
\includegraphics[scale=0.4]{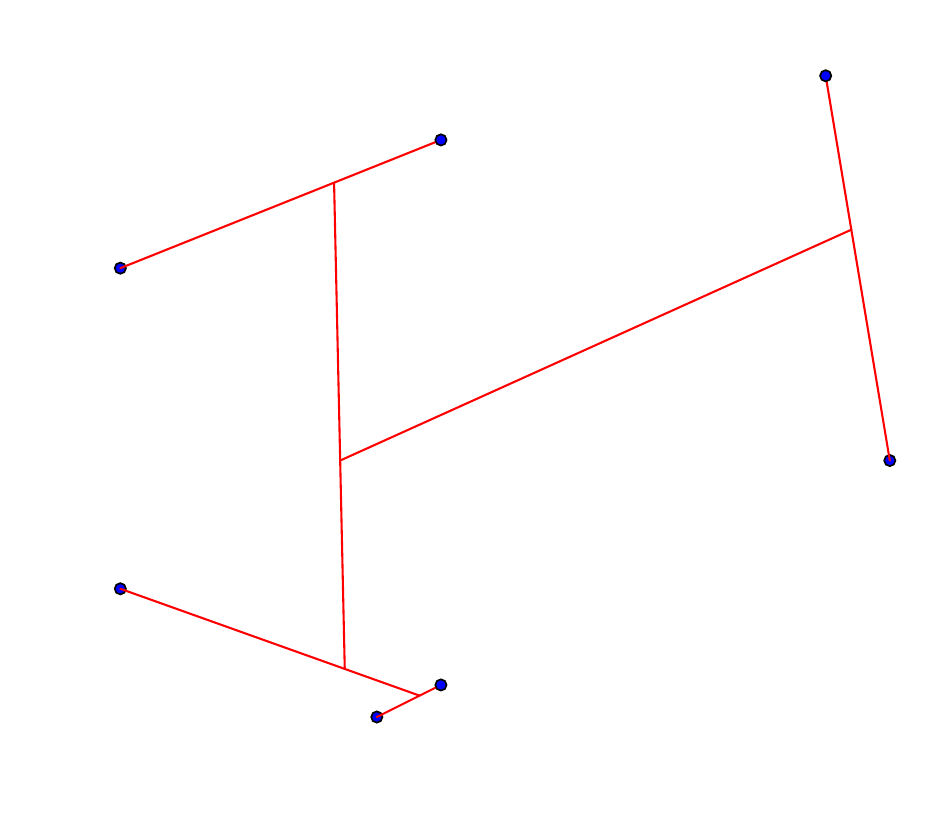}\caption{The hierarchical clustering configuration.\label{fig:HierNested}}
\par\end{centering}
\end{figure}
The choice of coordinates must capture the relative distances,
resolving both the macroscopic inter-cluster distances and the microscopic
intra-cluster scales, while ensuring the transformation remains canonical.
To illustrate this, let us consider the most simple '2-2' case on
the whole plane with its center of vorticity at the origin. One can
observe from Fig \ref{fig:Hier2-2} that $z_{1},z_{2}$ and $z_{3},z_{4}$ form two
clusters. \begin{figure}
\centering{}\includegraphics[scale=0.4]{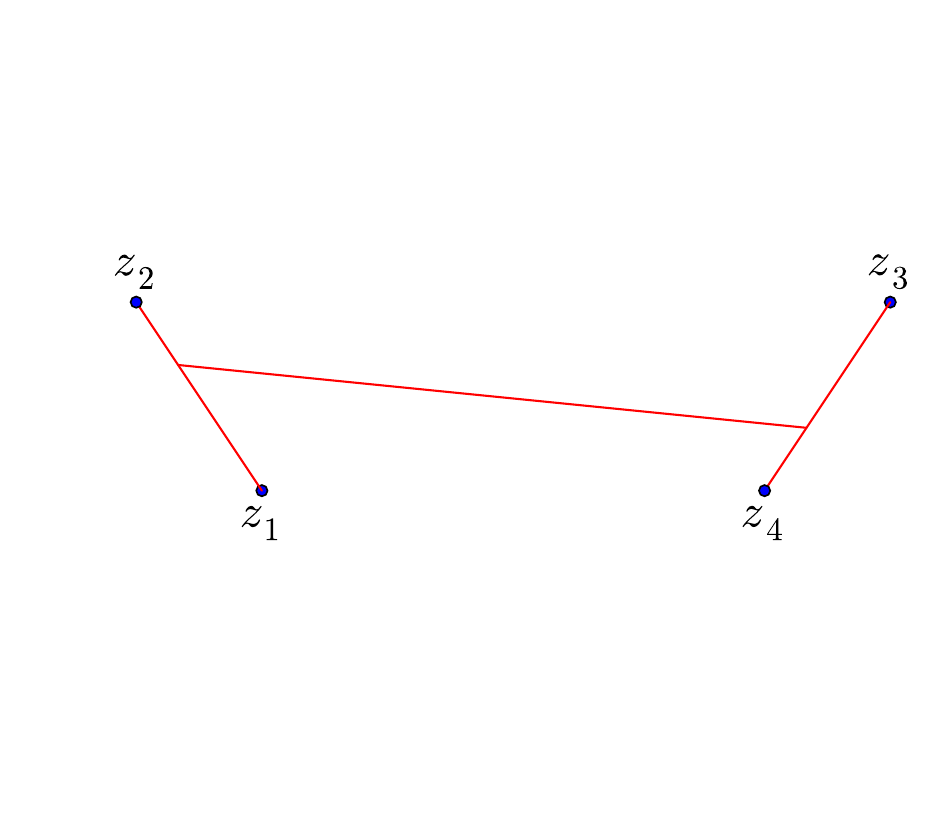}\caption{The '2-2' hierarchical clustering configuration. \label{fig:Hier2-2}}
\end{figure}
Let 
\begin{align*}
w_{1} & =z_{1}-z_{2},\\
w_{2} & =\frac{a_{1}z_{1}+a_{2}z_{2}}{a_{1}+a_{2}}-\frac{a_{3}z_{3}+a_{4}z_{4}}{a_{3}+a_{4}},\\
w_{3} & =z_{3}-z_{4},\\
w_{4} & =\frac{a_{1}z_{1}+a_{2}z_{2}+a_{3}z_{3}+a_{4}z_{4}}{a_{1}+a_{2}+a_{3}+a_{4}}\equiv0,
\end{align*}
where we assume that all the vortex strengths $a_{i}$ are chosen
properly so that all $w_{i}$ make sense. One can easily verify that
\[
\sum^{4}_{k=1}a_{k}dx_{k}\wedge dy_{k}=\frac{i}{2}\sum^{4}_{k=1}a_{k}dz_{k}\wedge d\bar{z}_{k}=\frac{i}{2}\sum^{4}_{k=1}\mu_{k}dw_{k}\wedge d\bar{w}_{k},
\]
where 
\begin{align*}
\mu_{1} & =\frac{a_{1}a_{2}}{a_{1}+a_{2}},\\
\mu_{2} & =\frac{\left(a_{1}+a_{2}\right)\left(a_{3}+a_{4}\right)}{a_{1}+a_{2}+a_{3}+a_{4}},\\
\mu_{3} & =\frac{a_{3}a_{4}}{a_{3}+a_{4}},\\
\mu_{4} & =a_{1}+a_{2}+a_{3}+a_{4}
\end{align*}
are called the reduced vorticity strengths. Hence the transform from
$\boldsymbol{z}$ to \textbf{$\boldsymbol{w}$ }is canonical and the
system's Hamiltonian in terms of coordinates $\boldsymbol{w}$ is
\begin{align*}
H\left(\boldsymbol{w}\right) & =-\sum_{1\le k<l\le4}\frac{a_{k}a_{l}}{4\pi}\ln\left(\left|z_{k}-z_{l}\right|^{2}\right)\\
 & =-\frac{a_{1}a_{2}}{4\pi}\ln\left(\left|w_{1}\right|^{2}\right)-\frac{a_{3}a_{4}}{4\pi}\ln\left(\left|w_{3}\right|^{2}\right)\\
 & \quad-\frac{a_{1}a_{3}}{4\pi}\ln\left(\left|w_{2}+\frac{a_{2}}{a_{1}+a_{2}}w_{1}-\frac{a_{4}}{a_{3}+a_{4}}w_{3}\right|^{2}\right)\\
 & \quad-\frac{a_{1}a_{4}}{4\pi}\ln\left(\left|w_{2}+\frac{a_{2}}{a_{1}+a_{2}}w_{1}+\frac{a_{3}}{a_{3}+a_{4}}w_{3}\right|^{2}\right)\\
 & \quad-\frac{a_{2}a_{3}}{4\pi}\ln\left(\left|w_{2}-\frac{a_{1}}{a_{1}+a_{2}}w_{1}-\frac{a_{4}}{a_{3}+a_{4}}w_{3}\right|^{2}\right)\\
 & \quad-\frac{a_{2}a_{4}}{4\pi}\ln\left(\left|w_{2}-\frac{a_{1}}{a_{1}+a_{2}}w_{1}+\frac{a_{3}}{a_{3}+a_{4}}w_{3}\right|^{2}\right).
\end{align*}
Since we assume that $\left|w_{2}\right|\gg\left|w_{1}\right|,\left|w_{3}\right|$,
one can expand the last four log terms and obtain the following expansion
\begin{align*}
H\left(\boldsymbol{w}\right) & =\underbrace{-\frac{a_{1}a_{2}}{4\pi}\ln\left(\left|w_{1}\right|^{2}\right)-\frac{a_{3}a_{4}}{4\pi}\ln\left(\left|w_{3}\right|^{2}\right)-\frac{\sum_{1\le k<l\le4}a_{k}a_{l}}{4\pi}\ln\left(\left|w_{2}\right|^{2}\right)}_{h_{0}}\\
 & \quad+\sum_{k+l\ge1}c_{k,l}\left(\frac{w_{1}}{w_{2}}\right)^{k}\left(\frac{w_{3}}{w_{2}}\right)^{l}+c_{k,l}\left(\frac{\bar{w}_{1}}{\bar{w}_{2}}\right)^{k}\left(\frac{\bar{w}_{3}}{\bar{w}_{2}}\right)^{l}.
\end{align*}
With the above form, we can follow a similar normal form procedure
in the main Theorem and prove a similar long time stability result
if initially $\left|\frac{w_{1}}{w_{2}}\right|=O\left(\varepsilon\right),\left|\frac{w_{3}}{w_{2}}\right|=O\left(\varepsilon\right)$.
We emphasize that there is no mandatory comparison between $\left|w_{1}\right|$
and $\left|w_{2}\right|$. 

\subsection{Multiple clusters configurations}

    We consider a configuration of three point vortices in which one vortex is replaced by a nearby pair of vortices. The pair acts as an effective single vortex, yielding a four-vortex system whose initial configuration is close to an equilateral triangle. (see Fig \ref{fig:tri1-1-2}).
\begin{figure}
\centering{}\includegraphics[scale=0.4]{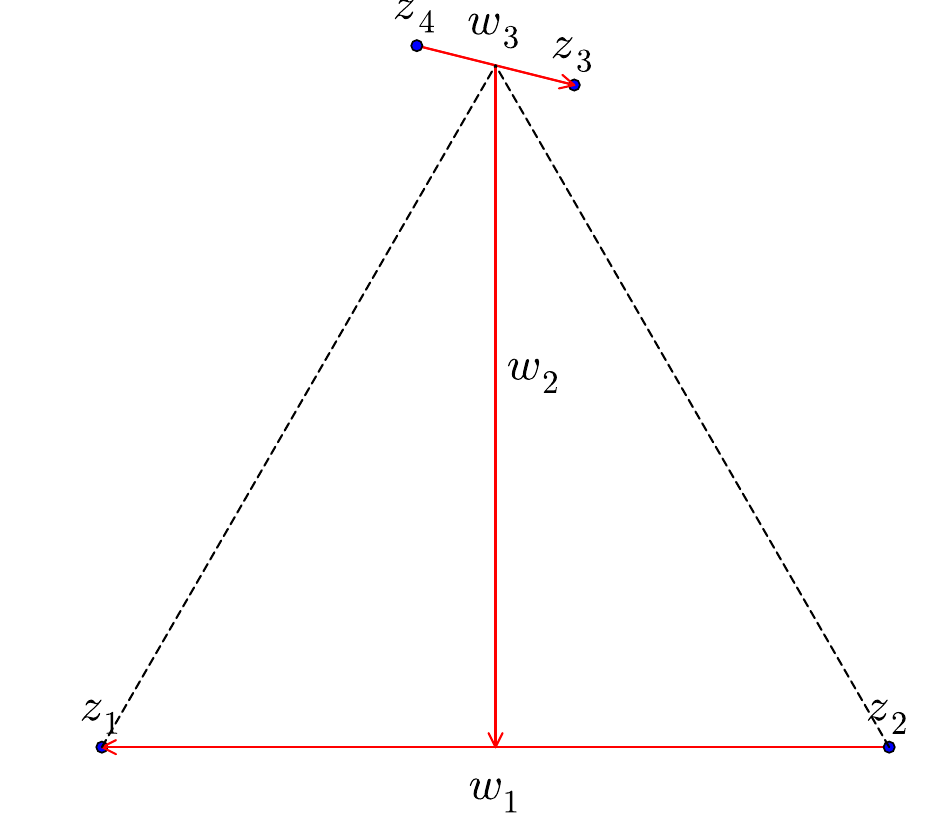}\caption{The '1+1-2' configuration near a stable relative equilibrium. \label{fig:tri1-1-2}}
\end{figure}
We call this configuration
as the '1+1-2' case. For simplicity we set all the vortex strengths
to be $1$. Let 
\begin{align*}
w_{1} & =z_{1}-z_{2},\\
w_{2} & =\frac{a_{1}z_{1}+a_{2}z_{2}}{a_{1}+a_{2}}-\frac{a_{3}z_{3}+a_{4}z_{4}}{a_{3}+a_{4}},\\
w_{3} & =z_{3}-z_{4},\\
w_{4} & =\frac{a_{1}z_{1}+a_{2}z_{2}+a_{3}z_{3}+a_{4}z_{4}}{a_{1}+a_{2}+a_{3}+a_{4}}\equiv0,
\end{align*}
where we set the center of vorticity at the origin. It can be verified
that the transform from $\boldsymbol{z}$ to $\boldsymbol{w}$ is
canonical. The Hamiltonian of this system is 
\begin{align*}
H & =-\frac{1}{4\pi}\left(\ln\left|w_{1}\right|^{2}+\ln\left|w_{3}\right|^{2}+\ln\left|w_{2}+\frac{w_{1}}{2}-\frac{w_{3}}{2}\right|^{2}+\ln\left|w_{2}-\frac{w_{1}}{2}+\frac{w_{3}}{2}\right|^{2}\right)\\
 & \quad-\frac{1}{4\pi}\left(\ln\left|w_{2}+\frac{w_{1}}{2}+\frac{w_{3}}{2}\right|^{2}+\ln\left|w_{2}-\frac{w_{1}}{2}-\frac{w_{3}}{2}\right|^{2}\right)\\
 & =-\frac{1}{4\pi}\left(\ln\left|w_{1}\right|^{2}+\ln\left|w_{3}\right|^{2}+2\ln\left|w_{2}+\frac{w_{1}}{2}\right|^{2}+2\ln\left|w_{2}-\frac{w_{1}}{2}\right|^{2}\right)\\
 & \quad-\frac{1}{4\pi}\left(\ln\left|1-\frac{w_{3}/2}{w_{2}+w_{1}/2}\right|^{2}+\ln\left|1+\frac{w_{3}/2}{w_{2}-w_{1}/2}\right|^{2}\right)\\
 & \quad-\frac{1}{4\pi}\left(\ln\left|1+\frac{w_{3}/2}{w_{2}+w_{1}/2}\right|^{2}+\ln\left|1-\frac{w_{3}/2}{w_{2}-w_{1}/2}\right|^{2}\right)\\
 & :=h_{w_{1},w_{2}}-\frac{1}{4\pi}\left(\ln\left|w_{3}\right|^{2}+\ln\left|1-\frac{w_{3}/2}{w_{2}+w_{1}/2}\right|^{2}+\ln\left|1+\frac{w_{3}/2}{w_{2}-w_{1}/2}\right|^{2}\right)\\
 & \qquad-\frac{1}{4\pi}\left(\ln\left|1+\frac{w_{3}/2}{w_{2}+w_{1}/2}\right|^{2}+\ln\left|1-\frac{w_{3}/2}{w_{2}-w_{1}/2}\right|^{2}\right).
\end{align*}
We remark there $h_{w_{1},w_{2}}$ is exactly the Hamiltonian of the
three point vortex system with vortex strengths $a_{1}=1,a_{2}=1,a_{3}=2$,
which precisely represents the case where the cluster is replaced
by an effective point vortex. Assuming that $z_{1},z_{2}$ and the
cluster $\left(z_{3},z_{4}\right)$ are located near a equilateral
triangle with the side length $L=1$. Let $\boldsymbol{u}$ be the
variable that filters out the rotation 
\[
\boldsymbol{w}=\boldsymbol{u}e^{i\Omega t},
\]
where $\Omega=\frac{\sum^{4}_{k=1}a_{k}}{2\pi L^{2}}=\frac{2}{\pi}$
is the frequency at which that equilateral triangle rotates. The Hamiltonian becomes 
\[
H_{rot}=\frac{\Omega}{2}\sum^{3}_{k=1}\mu_{k}\left|u_{k}\right|^{2}+H,
\]
where $\mu_{i}$ are the reduced vorticity strengths defined as 
\begin{align*}
\mu_{1} & =\frac{a_{1}a_{2}}{a_{1}+a_{2}}=\frac{1}{2}\\
\mu_{2} & =\frac{\left(a_{1}+a_{2}\right)\left(a_{3}+a_{4}\right)}{a_{1}+a_{2}+a_{3}+a_{4}}=1,\\
\mu_{3} & =\frac{a_{3}a_{4}}{a_{3}+a_{4}}=\frac{1}{2}.
\end{align*}
We expect the cluster and other two points to stay close to the equilibrium
, so one has $u_{1}\approx-L,u_{2}\approx-\lambda L$ with $\lambda=\frac{a_{1}-a_{2}}{2\left(a_{1}+a_{2}\right)}+i\frac{\sqrt{3}}{2}=i\frac{\sqrt{3}}{2}$.
More precisely, let 
\begin{align*}
u_{1} & =-L+\delta_{1},\\
u_{2} & =-\lambda L+\delta_{2},\\
u_{3} & =\delta_{3},
\end{align*}
and consider the system in terms of the perturbation $\left(\delta_{1},\delta_{2},\delta_{3}\right)$.
Assume that after the transformation $w_{1},w_{2}\to\delta_{1},\delta_{2}$,
$h_{w_{1},w_{2}}$ becomes $h_{\delta_{1},\delta_{2}}$. The new Hamiltonian
is 
{\small\begin{align*}
H_{\delta_{1},\delta_{2},\delta_{3}} & =\frac{\Omega}{2}\left(\mu_{1}\left|\delta_{1}-L\right|^{2}+\mu_{2}\left|\delta_{2}-\lambda L\right|^{2}\right)+h_{\delta_{1},\delta_{2}}-\frac{1}{4\pi}\ln\left|\delta_{3}\right|^{2}\\
 & \quad-\frac{1}{4\pi}\left(\ln\left|1-\frac{\delta_{3}/2}{-\lambda L+\delta_{2}+\frac{-L+\delta_{1}}{2}}\right|^{2}+\ln\left|1+\frac{\delta_{3}/2}{-\lambda L+\delta_{2}-\frac{-L+\delta_{1}}{2}}\right|^{2}\right)\\
 & \quad-\frac{1}{4\pi}\left(\ln\left|1+\frac{\delta_{3}/2}{-\lambda L+\delta_{2}+\frac{-L+\delta_{1}}{2}}\right|^{2}+\ln\left|1-\frac{\delta_{3}/2}{-\lambda L+\delta_{2}-\frac{-L+\delta_{1}}{2}}\right|^{2}\right).
\end{align*}}
Due to the fact that $\delta_{1},\delta_{2},\delta_{3}\ll L$, we expand
the last four $\log$ terms and obtain that 
\begin{align}
H_{\delta_{1},\delta_{2},\delta_{3}} & =\frac{\Omega}{2}\left(\mu_{1}\left|\delta_{1}-L\right|^{2}+\mu_{2}\left|\delta_{2}-\lambda L\right|^{2}\right)+h_{\delta_{1},\delta_{2}}\nonumber \\
 & \quad-\frac{1}{4\pi}\ln\left|\delta_{3}\right|^{2}-\frac{\delta^{2}_{3}}{16\pi}-\frac{\bar{\delta}^{2}_{3}}{16\pi}+O\left(\delta^{3}_{1}+\delta^{3}_{2}+\delta^{3}_{3}\right)\nonumber \\
 & :=H_{\delta_{1},\delta_{2}}-\frac{1}{4\pi}\ln\left|\delta_{3}\right|^{2}-\frac{\delta^{2}_{3}}{16\pi}-\frac{\bar{\delta}^{2}_{3}}{16\pi}+O\left(\delta^{3}_{1}+\delta^{3}_{2}+\delta^{3}_{3}\right)\label{eq:H_delta_4}
\end{align}
Now observe that $H_{\delta_{1},\delta_{2}}$ can be thought of as the Hamiltonian of the three point vortex system with strengths $a_{1}=1,a_{2}=1,a_{3}=2$ in a new Jacobi coordinate system.
Because the three point vortex system is completely integrable, there
exists a symplectic mapping $\left(\delta_{1},\delta_{2}\right)=\varphi\left(\sigma_{1},\sigma_{2}\right)$
such that the Hamiltonian in $\left(\sigma_{1},\sigma_{2}\right)$
coordinate becomes the following: 
\[
H_{\sigma_{1},\sigma_{2}}:=H_{\delta_{1},\delta_{2}}\circ\varphi=C+c_{1}\left|\sigma_{1}\right|^{2}+c_{2}\left|\sigma_{2}\right|^{2}+r\left(\left|\sigma_{1}\right|,\left|\sigma_{2}\right|\right)
\]
which is a function of $\left|\sigma_{1}\right|,\left|\sigma_{2}\right|$
only and the remainder $r$ is the higher order part. Applying this
symplectic mapping $\varphi$ to the clusters case, let 
\[
\left(\delta_{1},\delta_{2}\right)=\varphi\left(\sigma_{1},\sigma_{2}\right),\delta_{3}=\sigma_{3}
\]
 and the Hamiltonian \eqref{eq:H_delta_4} now becomes 
\begin{align*}
H_{\sigma_{1},\sigma_{2},\sigma_{3}} & =\underbrace{C+c_{1}\left|\sigma_{1}\right|^{2}+c_{2}\left|\sigma_{2}\right|^{2}-\frac{1}{4\pi}\ln\left|\sigma_{3}\right|^{2}}_{h_{0}}\\
 & \quad-\frac{\sigma^{2}_{3}}{16\pi}-\frac{\bar{\sigma}^{2}_{3}}{16\pi}+O\left(\sigma^{3}_{1}+\sigma^{3}_{2}+\sigma^{3}_{3}\right).
\end{align*}
With the above form we can perform a normal form and prove the long
time stability. 

In the general case, as long as all the clusters as effective point
vortices are near a stable configuration, there exists a local symplectic
mapping such that, in the new coordinates, the Hamiltonian takes a
form suitable for normal form reduction.
\section*{Acknowledgments}
S.I. acknowledges financial support from the Natural Sciences and Engineering Research Council of Canada (NSERC) under Grant No. 371637-2025. S.S. was supported by the National Natural Science Foundation of China (Grant No. 12671275) and the Ningbo Natural Science Foundation (Project No. 2025J005).

Part of this work was presented and discussed at the Fields Institute for Research in Mathematical Sciences during the Thematic Program on Shocks and Singularities: Nonlinear Evolution Equations in Physical and Life Sciences. We would like to thank the organizers and the staff of the Fields Institute for their hospitality and for providing a stimulating environment for collaboration.

\appendix
\section*{Appendix}
\section{Algebraic structure of the series expansion}
The positions of $N$ point vortices can be recovered explicitly through the Jacobi coordinates
\[
z_{k}=\left(\kappa_{k}-1\right)w_{k-1}+\sum^{N}_{j=k}\kappa_{j+1}w_{j},
\]
where the sum index $j$ is understood as an element in $\mathbb{Z}/N\mathbb{Z}$ and $\kappa_k = \frac{a_k}{A_k}$ with $A_k = \sum_{i=1}^k a_i$. This linear relation leads to the following special algebraic structure for the series expansion.
\begin{prop}
\label{prop:base_structure}Given an integer $p\ge1$, expand $Q_{p}=\sum^{N}_{k=1}a_{k}z^{p}_{k}$
as a power series of $\boldsymbol{w}$
\[
Q_{p}=\sum c_{\boldsymbol{K}}\boldsymbol{w}^{\boldsymbol{K}}.
\]
If $c_{\boldsymbol{K}}\neq0$ and $K_{l}$ with $1\le l\le N-1$ is
the first non-zero component of $\boldsymbol{K}$, then $K_{l}\neq1$. 
\end{prop}
\begin{proof}
According to the linear transform between $\boldsymbol{z}$ and $\boldsymbol{w}$,
\textbf{$w_{l}$} only appears in $z_{1},z_{2},\cdots,z_{l+1}$. So
such $\boldsymbol{w}^{\boldsymbol{K}}$ can only be found in 
\[
a_{1}z^{p}_{1}+a_{2}z^{p}_{2}+\cdots+a_{l+1}z^{p}_{l+1}.
\]
A straightforward calculation shows the coefficients of $\boldsymbol{w}^{\boldsymbol{K}}$in
$a_{j}z^{p}_{j},1\le j\le l+1$ are
\begin{align*}
\frac{a_{j}p!}{K_{l}!K_{l+1}!\cdots K_{N}!}\kappa^{K_{l}}_{l+1}\kappa^{K_{l+1}}_{l+2}\cdots\kappa^{K_{N-1}}_{N-1}, & 1\le j\le l\\
\frac{a_{l+1}p!}{K_{l}!K_{l+1}!\cdots K_{N}!}\left(\kappa_{l+1}-1\right)^{K_{l}}\kappa^{K_{l+1}}_{l+2}\cdots\kappa^{K_{N-1}}_{N-1}, & j=l+1.
\end{align*}
Thus one obtains 
\[
c_{\boldsymbol{K}}=\frac{p!}{K_{l}!K_{l+1}!\cdots K_{N}!}\kappa^{K_{l+1}}_{l+2}\cdots\kappa^{K_{N-1}}_{N-1}\left(A_{l}\kappa^{K_{l}}_{l+1}+a_{l+1}\left(\kappa_{l+1}-1\right)^{K_{l}}\right).
\]
If $K_{l}=1$, it leads to 
\[
A_{l}\kappa^{K_{l}}_{l+1}+a_{l+1}\left(\kappa_{l+1}-1\right)^{K_{l}}=A_{l+1}\kappa_{l+1}-a_{l+1}=0,
\]
which is a contradiction. 
\end{proof}
A similar result holds for the function composition.
\begin{prop}\label{prop:mix_structure}
Given $\varphi\left(z\right)=\sum_{k\ge1}c_{k}z^{k}$ and an integer
$p\ge1$, expand $F=\sum^{N}_{k=1}a_{k}\varphi\left(z_{k}\right)^{p}$
as a power series of $\boldsymbol{w}$
\[
F=\sum c_{\boldsymbol{K}}\boldsymbol{w}^{\boldsymbol{K}}
\]
If $c_{\boldsymbol{K}}\neq0$ and $K_{l}$ with $1\le l\le N-1$ is
the first non-zero component of $\boldsymbol{K}$, then $K_{l}\neq1$. 
\end{prop}
\begin{proof}
    Plugging the expansion of $\varphi$ gives 
\begin{align*}
F & =\sum^{N}_{k=1}\sum_{l\ge1}a_{k}\left(c_{l}z^{l}_{k}\right)^{p}\\
 & =\sum_{l\ge p}c_{l,p}\left(\sum^{N}_{k=1}a_{k}z^{l}_{k}\right).
\end{align*}
Applying Proposition \ref{prop:base_structure} to $\sum^{N}_{k=1}a_{k}z^{l}_{k}$ immediately
proves the result.
\end{proof}

\section{The multi-index inequality}
For an $N$-dimensional multi-index $\boldsymbol{K}$, define 
\[
\widehat{\boldsymbol{K}}=\left(0,K_{1},K_{2},\cdots,K_{N-2},0\right),
\]
and let $i_{\boldsymbol{K},m}$ be the index of the first non-zero
entry of 
\[
\Delta_{m}\boldsymbol{K}:=\begin{cases}
\boldsymbol{K}-\widehat{\boldsymbol{K}}+\boldsymbol{e}_{m}, & 1\le m\le N-1,\\
\boldsymbol{K}-\widehat{\boldsymbol{K}}, & m=N.
\end{cases}
\]
If all entries are $0$, set $i_{\boldsymbol{K},m}=\infty$. We have
the following Proposition which will be used in the deformation estimates
after applying a canonical transformation.
\begin{prop}
\label{prop:index_inequality}Let $1\le m\le N$ and assume that $\left(\boldsymbol{K},\boldsymbol{L}\right)\in\mathscr{R}_{m-1}\setminus\mathscr{R}_{m}$
such that both the supports of $\boldsymbol{K}\left[1:m\right],\boldsymbol{L}\left[1:m\right]$ are
connected. Then it holds that
\[
\left|\boldsymbol{K}\right|+\left|\boldsymbol{L}\right|-2m+2i_{\boldsymbol{K},m}\ge1.
\]
\end{prop}
\begin{proof}
For $1\le m\le N$, the $j$-th component of $\Delta_{m}\boldsymbol{K}$
is 
\[
\left(\Delta_{m}\boldsymbol{K}\right)_{j}=\begin{cases}
K_{j}-K_{j-1}+\delta_{jm}, & 1\le j\le N-1,\\
K_{N}, & j=N,
\end{cases}
\]
with the convention that $K_{0}=L_{0}=0$. Assuming that the index
of the first non-zero entry of $\boldsymbol{K}$ is $i_{\boldsymbol{K}}$.
For those $j<\min\left(i_{\boldsymbol{K}},m\right)$, it is clear
that $j\le N-1,K_{j}=K_{j-1}=0$ so that 
\[
\left(\Delta_{m}\boldsymbol{K}\right)_{j}=K_{j}-K_{j-1}+\delta_{jm}=0.
\]
If $j=\min\left(i_{\boldsymbol{K}},m\right)$ one has either $K_{j}=K_{i_{\boldsymbol{K}}}\ge1$
or $\delta_{jm}=\delta_{mm}=1$, and in both cases $K_{j-1}=0$. So
that if $j\le N-1$, it holds that
\[
\left(\Delta_{m}\boldsymbol{K}\right)_{j}=K_{j}+\delta_{jm}\ge1,
\]
while if $j=N$, one has 
\[
\left(\Delta_{m}\boldsymbol{K}\right)_{N}=K_{N}=K_{j}\ge1.
\]
Thus $i_{\boldsymbol{K},m}=\min\left(i_{\boldsymbol{K}},m\right)$
and the proof goes into the following two cases.

Case 1, $i_{\boldsymbol{K},m}=m\le i_{\boldsymbol{K}}$. In this case
$K_{j}=0,j=1,2,\cdots,m-1$. Since $\left(\boldsymbol{K},\boldsymbol{L}\right)\in\mathscr{R}_{m-1}\setminus\mathscr{R}_{m}$,
by definition one has $L_{j}=K_{j}=0,j=1,2,\cdots,m-1$ and $K_{m}\neq L_{m}$.
Then at least one of $K_{m},L_{m}$ is strictly positive so that $\left|\boldsymbol{K}\right|+\left|\boldsymbol{L}\right|\ge1$.
Plugging it into the target inequality yields
\[
\left|\boldsymbol{K}\right|+\left|\boldsymbol{L}\right|-2m+2i_{\boldsymbol{K},m}\ge1+2\left(i_{\boldsymbol{K}}-m\right)=1.
\]

Case 2, $i_{\boldsymbol{K},m}=i_{\boldsymbol{K}}<m$. By definition
one has $K_{i_{\boldsymbol{K}}}=L_{i_{\boldsymbol{K}}}\ge1$ and $i_{\boldsymbol{K}}$
is also the index of the first non-zero entry of $\boldsymbol{L}$.
Because $K_{m}\neq L_{m}$, at least one of $K_{m},L_{m}$ is strictly
positive. If $K_{m}\ge1$, since the support $\boldsymbol{K}\left[1:m\right]$ is
connect, one obtains $K_{j}\ge1$ for $j=i_{\boldsymbol{K}},i_{\boldsymbol{K}}+1,\cdots,m$
so that $\left|\boldsymbol{K}\right|\ge m-i_{\boldsymbol{K}}+1$.
Meanwhile $L_{j}=K_{j}\ge1$ for $j=i_{\boldsymbol{K}},i_{\boldsymbol{K}}+1,\cdots,m-1$,
one also has $\left|\boldsymbol{L}\right|\ge m-i_{\boldsymbol{K}}$.
Plugging them into the target inequality we end up with 
\[
\left|\boldsymbol{K}\right|+\left|\boldsymbol{L}\right|-2m+2i_{\boldsymbol{K},m}\ge2m-2i_{\boldsymbol{K}}+1-2m+2i_{\boldsymbol{K}}=1.
\]
If $L_{m}\ge1$, similar argument applies to $\boldsymbol{L}$ leads
to the same result. 
\end{proof}
\section{Proof of Lemma \ref{claim:oder_increase}}\label{sec:proof_oder_increase}
Assuming that $\left(\boldsymbol{K},\boldsymbol{L}\right)\in\mathscr{C},f=\boldsymbol{\beta}^{\boldsymbol{K}'}\bar{\boldsymbol{\beta}}^{\boldsymbol{L'}}$
and $\left|\boldsymbol{K}\right|+\left|\boldsymbol{L}\right|=k\ge1,\left|\boldsymbol{K}'\right|+\left|\boldsymbol{L}'\right|=k'$,
we consider the lowest order of $\left\{ \boldsymbol{\beta}^{\boldsymbol{K}+\boldsymbol{E}_{m}}\bar{\boldsymbol{\beta}}^{\boldsymbol{L }+\boldsymbol{E}_{m}},f\right\} $
for each $\left(\boldsymbol{K},\boldsymbol{L}\right)\in\mathscr{R}_{m-1}\setminus\mathscr{R}_{m},1\le m\le N$
. By \eqref{eq:poly_action}, for $1\le m\le N-1$, it holds that
{\small
\begin{align*}
 & \left\{ \boldsymbol{\beta}^{\boldsymbol{K}+\boldsymbol{E}_{m}}\bar{\boldsymbol{\beta}}^{\boldsymbol{L}+\boldsymbol{E}_{m}},\boldsymbol{\beta}^{\boldsymbol{K}'}\bar{\boldsymbol{\beta}}^{\boldsymbol{L'}}\right\} _{\mathcal{S}}\\
= & \sum^{N-1}_{l=1}c_{l}\left(\delta_{l}\left(\boldsymbol{K}+\boldsymbol{E}_{m}\right)\delta_{l}\left(\boldsymbol{L}'\right)-\delta_{l}\left(\boldsymbol{L}+\boldsymbol{E}_{m}\right)\delta_{l}\left(\boldsymbol{K}'\right)\right)\boldsymbol{\beta}^{\boldsymbol{K}+\boldsymbol{K}'+\boldsymbol{E}_{m}-\boldsymbol{E}_{l}}\bar{\boldsymbol{\beta}}^{\boldsymbol{L}+\boldsymbol{L}'+\boldsymbol{E}_{m}-\boldsymbol{E}_{l}}\\
 & +c_{N}\left(\delta_{N}\left(\boldsymbol{K}\right)\delta_{N}\left(\boldsymbol{L}'\right)-\delta_{N}\left(\boldsymbol{L}\right)\delta_{N}\left(\boldsymbol{K}'\right)\right)\boldsymbol{\beta}^{\boldsymbol{K}+\boldsymbol{K}'+\boldsymbol{E}_{m}-\boldsymbol{e}_{N}}\bar{\boldsymbol{\beta}}^{\boldsymbol{L}+\boldsymbol{L}'+\boldsymbol{E}_{m}-\boldsymbol{e}_{N}},
\end{align*}
}and for $m=N$, 
\begin{align*}
 & \left\{ \boldsymbol{\beta}^{\boldsymbol{K}}\bar{\boldsymbol{\beta}}^{\boldsymbol{L}},\boldsymbol{\beta}^{\boldsymbol{K}'}\bar{\boldsymbol{\beta}}^{\boldsymbol{L'}}\right\} _{\mathcal{S}}\\
= & \sum^{N-1}_{l=1}c_{l}\left(\delta_{l}\left(\boldsymbol{K}\right)\delta_{l}\left(\boldsymbol{L}'\right)-\delta_{l}\left(\boldsymbol{L}\right)\delta_{l}\left(\boldsymbol{K}'\right)\right)\boldsymbol{\beta}^{\boldsymbol{K}+\boldsymbol{K}'-\boldsymbol{E}_{l}}\bar{\boldsymbol{\beta}}^{\boldsymbol{L}+\boldsymbol{L}'-\boldsymbol{E}_{l}}\\
 & +c_{N}\left(\delta_{N}\left(\boldsymbol{K}\right)\delta_{N}\left(\boldsymbol{L}'\right)-\delta_{N}\left(\boldsymbol{L}\right)\delta_{N}\left(\boldsymbol{K}'\right)\right)\boldsymbol{\beta}^{\boldsymbol{K}+\boldsymbol{K}'-\boldsymbol{e}_{N}}\bar{\boldsymbol{\beta}}^{\boldsymbol{L}+\boldsymbol{L}'-\boldsymbol{e}_{N}}.
\end{align*}
We count the degree for every term in the summation above.
\begin{itemize}
\item $1\le m\le N-1,m\le l\le N-1$. Then $\left|\boldsymbol{E}_{m}-\boldsymbol{E}_{l}\right|=l-m\ge0$
and immediately we have 
\[
\left|\boldsymbol{K}+\boldsymbol{K}'+\boldsymbol{E}_{m}-\boldsymbol{E}_{l}\right|+\text{\ensuremath{\left|\boldsymbol{L}+\boldsymbol{L}'+\boldsymbol{E}_{m}-\boldsymbol{E}_{l}\right|}}\ge k'+1.
\]
\item $1\le m\le N-1,l=N$. Then $\left|\boldsymbol{E}_{m}-\boldsymbol{e}_{N}\right|=N-m-1\ge0$
and it follows 
\[
\left|\boldsymbol{K}+\boldsymbol{K}'+\boldsymbol{E}_{m}-\boldsymbol{E}_{l}\right|+\text{\ensuremath{\left|\boldsymbol{L}+\boldsymbol{L}'+\boldsymbol{E}_{m}-\boldsymbol{E}_{l}\right|}}\ge k'+1.
\]
\item $1\le m\le N-1,l<m$. Then $\left|\boldsymbol{E}_{m}-\boldsymbol{E}_{l}\right|=l-m<0$.
Let $i_{\boldsymbol{K}}$ be the index such that $K_{i_{\boldsymbol{K}}-1}=0,K_{i_{\boldsymbol{K}}}\neq0$.
It is clear that $i_{\boldsymbol{K}}$ exists and is unique if $\boldsymbol{K}\neq\boldsymbol{0}$. Meanwhile, since $\left(\boldsymbol{K},\boldsymbol{L}\right)\in\mathscr{R}_{m-1}\setminus\mathscr{R}_{m}$,
we have 
\[
K_{l}=L_{l},K_{l+1}=L_{l+1},\cdots,K_{m-1}=L_{m-1},K_{m}\neq L_{m}.
\]
Assuming that $K_{m}\neq0$, then $i_{\boldsymbol{K}}\le m$. Comparing
the value of $i_{\boldsymbol{K}}$ and $l$ yields
\begin{itemize}
\item $i_{\boldsymbol{K}}\le l$. Because $\left(\boldsymbol{K},\boldsymbol{L}\right)\in\mathscr{C}$,
the support of $\boldsymbol{K}\left[1:m\right]$ is connected. So that 
\[
K_{l}=L_{l}\ge1,K_{l+1}=L_{l+1}\ge1,\cdots,K_{m-1}=L_{m-1}\ge1,
\]
which implies $\left|\boldsymbol{K}\right|\ge m-l+1$ and $\left|\boldsymbol{L}\right|\ge m-l$.
Therefore, 
\begin{align*}
 & \left|\boldsymbol{K}+\boldsymbol{K}'+\boldsymbol{E}_{m}-\boldsymbol{E}_{l}\right|+\text{\ensuremath{\left|\boldsymbol{L}+\boldsymbol{L}'+\boldsymbol{E}_{m}-\boldsymbol{E}_{l}\right|}}\\
\ge & 2\left(m-l\right)+1+k'+2\left(l-m\right)\\
= & k'+1.
\end{align*}
\item $i_{\boldsymbol{K}}>l$. In this case one has $K_{l}=L_{l}=0$. Again
because $\left(\boldsymbol{K},\boldsymbol{L}\right)\in\mathscr{C}$,
the support of $\boldsymbol{K}\left[1:m\right]$ is connected, we have $K_{l-1}=L_{l-1}=0$.
Since $l<m$, it follows that 
\[
\delta_{l}\left(\boldsymbol{K}+\boldsymbol{E}_{m}\right)=\delta_{l}\left(\boldsymbol{L}+\boldsymbol{E}_{m}\right)=\delta_{l}\left(\boldsymbol{K}\right)=\delta_{l}\left(\boldsymbol{L}\right)=0.
\]
Then the coefficient of $\boldsymbol{\beta}^{\boldsymbol{K}+\boldsymbol{K}'+\boldsymbol{E}_{m}-\boldsymbol{E}_{l}}\bar{\boldsymbol{\beta}}^{\boldsymbol{L}+\boldsymbol{L}'+\boldsymbol{E}_{m}-\boldsymbol{E}_{l}}$
is 
\[
\delta_{l}\left(\boldsymbol{K}+\boldsymbol{E}_{m}\right)\delta_{l}\left(\boldsymbol{L}'\right)-\delta_{l}\left(\boldsymbol{L}+\boldsymbol{E}_{m}\right)\delta_{l}\left(\boldsymbol{K}'\right)=0.
\]
The term vanishes!
\end{itemize}
\item $m=N,l<N$. By noting that $\boldsymbol{E}_{N}=\boldsymbol{0}$, we
follow exactly the same discussion with the case $1\le m\le N-1,l<m$
and obtain that either the degree is greater or equal to $k'+1$ or
the term vanishes. 
\item $m=N,l=N$. Since $\left(\boldsymbol{K},\boldsymbol{L}\right)\in\mathscr{R}_{N-1}\setminus\mathscr{R}_{N}$,
we have $K_{N}\neq L_{N}$. Let us assume $K_{N}\neq0$. Again we
compare $i_{\boldsymbol{K}}$ and $N$.
\begin{itemize}
\item $i_{\boldsymbol{K}}<N$. Since the support of $\boldsymbol{K}$ is
connect. Then one has 
\[
K_{i_{\boldsymbol{K}}}=L_{i_{\boldsymbol{K}}}\ge1,\cdots,K_{N-1}=L_{N-1}\ge1.
\]
So that it holds $\left|\boldsymbol{K}\right|\ge N-i_{\boldsymbol{K}}+1$
and $\left|\boldsymbol{L}\right|\ge N-i_{\boldsymbol{K}}$. It follows
\begin{align*}
 & \left|\boldsymbol{K}+\boldsymbol{K}'-\boldsymbol{e}_{N}\right|+\text{\ensuremath{\left|\boldsymbol{L}+\boldsymbol{L}'-\boldsymbol{e}_{N}\right|}}\\
\ge & 2\left(N-i_{\boldsymbol{K}}\right)+1+k'-2\\
\ge & k'+1.
\end{align*}
\item $i_{\boldsymbol{K}}=N$. In this case only $K_{N},L_{N}$ survive.
According to the perturbation's structure \eqref{eq:H0}, such terms
only appear when the degree is at least $3$, that is, $\left|\boldsymbol{K}\right|+\left|\boldsymbol{L}\right|\ge3$.
So that we have 
\[
\left|\boldsymbol{K}+\boldsymbol{K}'-\boldsymbol{e}_{N}\right|+\text{\ensuremath{\left|\boldsymbol{L}+\boldsymbol{L}'-\boldsymbol{e}_{N}\right|}}\ge3+k'-2=k'+1.
\]
\end{itemize}
\end{itemize}
Summarizing the above discussion, we conclude that for $\left(\boldsymbol{K},\boldsymbol{L}\right)\in\mathscr{C}$ and $K_N = L_N = 0$ if $\left|\boldsymbol{K}\right|+\left|\boldsymbol{L}\right|\le2$,
\begin{equation}
\mathrm{ord}\left(\left\{ \boldsymbol{\beta}^{\boldsymbol{K}+\boldsymbol{E}_{m}}\bar{\boldsymbol{\beta}}^{\boldsymbol{L}+\boldsymbol{E}_{m}},f\right\} \right)\ge\mathrm{ord}\left(f\right)+1,\quad1\le m\le N.\label{eq:element_order_increase}
\end{equation}
Noting that by Proposition \ref{prop:full_sol_homological_eqn}, $g_{k}=\sum^{N}_{m=1}g_{k,m}$
and $g_{k,m}$ is the linear combination of 
\[
\boldsymbol{\beta}^{\boldsymbol{K}+\boldsymbol{E}_{m}}\bar{\boldsymbol{\beta}}^{\boldsymbol{L}+\boldsymbol{E}_{m}}
\]
with $\left(\boldsymbol{K},\boldsymbol{L}\right)\in\mathscr{R}_{m-1}\setminus\mathscr{R}_{m}$.
Moreover, such $\left(\boldsymbol{K},\boldsymbol{L}\right)$ belongs
to $\mathscr{E}\left(p_{k-1}\right)\subset\mathscr{E}\left(R_{k-1}\right)\subset\mathscr{C}$.
Thus the Lemma is proved by applying \eqref{eq:element_order_increase}. 

\section{Proof of Lemma \ref{claim:remainder_structure}}\label{sec:proof_remainder_structure}

Considering the outcome of 
\[
\left\{ \boldsymbol{\beta}^{\boldsymbol{K}+\boldsymbol{E}_{m}}\bar{\boldsymbol{\beta}}^{\boldsymbol{L}+\boldsymbol{E}_{m}},\boldsymbol{\beta}^{\boldsymbol{K}'}\bar{\boldsymbol{\beta}}^{\boldsymbol{L'}}\right\} _{\mathcal{S}},
\]
we apply the formula \eqref{eq:poly_action} and check every term
in the summation.
These terms are 
\begin{itemize}
\item For $1\le m\le N,1\le l\le N-1$,
\begin{equation}
\left(\delta_{l}\left(\boldsymbol{K}+\boldsymbol{E}_{m}\right)\delta_{l}\left(\boldsymbol{L}'\right)-\delta_{l}\left(\boldsymbol{L}+\boldsymbol{E}_{m}\right)\delta_{l}\left(\boldsymbol{K}'\right)\right)\boldsymbol{\beta}^{\boldsymbol{K}+\boldsymbol{K}'+\boldsymbol{E}_{m}-\boldsymbol{E}_{l}}\bar{\boldsymbol{\beta}}^{\boldsymbol{L}+\boldsymbol{L}'+\boldsymbol{E}_{m}-\boldsymbol{E}_{l}}.\label{eq:action_l<N-1}
\end{equation}
\item For $1\le m\le N,l=N$,
\begin{equation}
\left(\delta_{N}\left(\boldsymbol{K}\right)\delta_{N}\left(\boldsymbol{L}'\right)-\delta_{N}\left(\boldsymbol{L}\right)\delta_{N}\left(\boldsymbol{K}'\right)\right)\boldsymbol{\beta}^{\boldsymbol{K}+\boldsymbol{K}'+\boldsymbol{E}_{m}-\boldsymbol{e}_{N}}\bar{\boldsymbol{\beta}}^{\boldsymbol{L}+\boldsymbol{L}'+\boldsymbol{E}_{m}-\boldsymbol{e}_{N}}.\label{eq:action_l=00003DN}
\end{equation}
\end{itemize}

By comparing $m,l$ all terms in \eqref{eq:action_l<N-1} and \eqref{eq:action_l=00003DN} are discussed in the following cases.
\begin{itemize}
    \item $1\le m \le l \le N-1$. Then 
\[
\begin{array}{cccccccccc}
\text{Index} & 1 & \cdots & m-1 & m & \cdots & l-1 & l & \cdots & N\\
\boldsymbol{E}_{m}-\boldsymbol{E}_{l}: & 0 & \cdots & 0 & 1 & \cdots & 1 & 0 & \cdots & 0
\end{array},
\]
thus it is obvious that 
\[
\left(\boldsymbol{K}+\boldsymbol{K}'+\boldsymbol{E}_{m}-\boldsymbol{E}_{l},\boldsymbol{L}+\boldsymbol{L}'+\boldsymbol{E}_{m}-\boldsymbol{E}_{l}\right)\in\mathscr{C}.
\]
\item $m\le l,l=N$. Then \eqref{eq:action_l=00003DN} survive if at least
one of $\boldsymbol{K}_{N},\boldsymbol{L}_{N}$ is non-zero. If both
of them are non-zero, it is again easy to see that 
\[
\left(\boldsymbol{K}+\boldsymbol{K}'+\boldsymbol{E}_{m}-\boldsymbol{e}_{N},\boldsymbol{L}+\boldsymbol{L}'+\boldsymbol{E}_{m}-\boldsymbol{e}_{N}\right)\in\mathscr{C}.
\]
Now let us assume $\boldsymbol{K}_{N}\neq0,\boldsymbol{L}_{N}=0$.
Then \eqref{eq:action_l=00003DN} becomes
\[
c_{N}\delta_{N}\left(\boldsymbol{K}\right)\delta_{N}\left(\boldsymbol{L}'\right)\boldsymbol{\beta}^{\boldsymbol{K}+\boldsymbol{K}'+\boldsymbol{E}_{m}-\boldsymbol{e}_{N}}\bar{\boldsymbol{\beta}}^{\boldsymbol{L}+\boldsymbol{L}'+\boldsymbol{E}_{m}-\boldsymbol{e}_{N}},
\]
and one needs $\boldsymbol{L}'_{N}\neq0$ to make it alive. So that
\[
\left(\boldsymbol{K}+\boldsymbol{K}'+\boldsymbol{E}_{m}-\boldsymbol{e}_{N},\boldsymbol{L}+\boldsymbol{L}'+\boldsymbol{E}_{m}-\boldsymbol{e}_{N}\right)\in\mathscr{C}
\]
holds true. The case $\boldsymbol{K}_{N}=0,\boldsymbol{L}_{N}\neq0$
leads to the same result.
\item $l<m$. This case only appears in \eqref{eq:action_l<N-1}. Let $i_{\boldsymbol{K}}$ be the index such that $K_{j}=0,1\le j < i_{\boldsymbol{K}},K_{i_{\boldsymbol{K}}}\neq0$. Comparing $i_{\boldsymbol{K}}$ and $l$ yields the following sub-cases.
\begin{itemize}
\item $l<i_{\boldsymbol{K}}$. Then $K_{l}=L_{l}=0$. Similar to the case
in the proof of Claim \ref{claim:oder_increase}, one has $K_{l-1}=L_{l-1}=0$
and 
\[
\delta_{l}\left(\boldsymbol{K}+\boldsymbol{E}_{m}\right)=\delta_{l}\left(\boldsymbol{L}+\boldsymbol{E}_{m}\right)=\delta_{l}\left(\boldsymbol{K}\right)=\delta_{l}\left(\boldsymbol{L}\right)=0.
\]
The term \eqref{eq:action_l<N-1} vanishes!
\item $i_{\boldsymbol{K}}\le l$. In this case, the value of $\boldsymbol{K}', \boldsymbol{L}', \boldsymbol{K}+\boldsymbol{E}_{m}-\boldsymbol{E}_{l}$
and $\boldsymbol{L}+\boldsymbol{E}_{m}-\boldsymbol{E}_{l}$ are 
\[
\begin{array}{cccccccc}
\text{Index} & \cdots & l-1 & l & \cdots & m-1 & m & \cdots\\
\boldsymbol{K}+\boldsymbol{E}_{m}-\boldsymbol{E}_{l}: & \cdots & K_{l-1} & K_{l}-1 & \cdots & K_{m-1}-1 & K_{m} & \cdots\\
\boldsymbol{L}+\boldsymbol{E}_{m}-\boldsymbol{E}_{l}: & \cdots & K_{l-1} & K_{l}-1 & \cdots & K_{m-1}-1 & L_{m} & \cdots \\
\boldsymbol{K}': & \cdots & K'_{l-1} & K'_l & \cdots \\
\boldsymbol{L}': & \cdots & L'_{l-1} & L'_l & \cdots
\end{array}.
\]
Assume the term \eqref{eq:action_l<N-1} exists. By the fact that $\left(\boldsymbol{K}, \boldsymbol{L}\right) \in \mathscr{C}$, one has clearly that 
\[
K_j-1 \ge 1, \quad l\le j \le m-1
\]
so that 
\[
supp\left(\boldsymbol{K}+\boldsymbol{K}'+\boldsymbol{E}_{m}-\boldsymbol{E}_{l}  + \boldsymbol{L}+\boldsymbol{L}'+\boldsymbol{E}_{m}-\boldsymbol{E}_{l} \right)
\]
is connected.
If 
\[
\left(\boldsymbol{K}+\boldsymbol{K}'+\boldsymbol{E}_{m}-\boldsymbol{E}_{l},\boldsymbol{L}+\boldsymbol{L}'+\boldsymbol{E}_{m}-\boldsymbol{E}_{l}\right)\in\mathscr{R}_{N},
\]
 it is clear that $\left(\boldsymbol{K}+\boldsymbol{K}'+\boldsymbol{E}_{m}-\boldsymbol{E}_{l},\boldsymbol{L}+\boldsymbol{L}'+\boldsymbol{E}_{m}-\boldsymbol{E}_{l}\right)\in\mathscr{C}$ by definition. If
\[
\left(\boldsymbol{K}+\boldsymbol{K}'+\boldsymbol{E}_{m}-\boldsymbol{E}_{l},\boldsymbol{L}+\boldsymbol{L}'+\boldsymbol{E}_{m}-\boldsymbol{E}_{l}\right)\in\mathscr{R}_{\tilde{m}-1}\setminus\mathscr{R}_{\tilde{m}}
\]
where $1\le \tilde{m} \le N$, an observation is that $K'_l = L'_l, K'_{l-1} = L'_{l-1}$ can not happen simultaneously otherwise the term \eqref{eq:action_l<N-1} vanishes. Then one must have $\left(\boldsymbol{K}', \boldsymbol{L}'\right) \in \mathscr{R}_{m'-1}\setminus\mathscr{R}_{m'}$ for some $1\le m'\le l$. Since $l<m$, actually it holds that $\tilde{m}=m'\le l < m$. Because $\left(\boldsymbol{K}, \boldsymbol{L}\right), \left(\boldsymbol{K}', \boldsymbol{L}'\right) \in \mathscr{C}$, we have 
\[
K_j \neq 1, K'_j\neq 1, \quad 1\le j \le \tilde{m}-1.
\]
Therefore we conclude that 
\[
\left(\boldsymbol{K}+\boldsymbol{K}'+\boldsymbol{E}_{m}-\boldsymbol{E}_{l},\boldsymbol{L}+\boldsymbol{L}'+\boldsymbol{E}_{m}-\boldsymbol{E}_{l}\right)\in\mathscr{C}.
\]
\end{itemize}
\end{itemize}
Thus the Lemma is proved by summarizing all the discussion above.
\bibliographystyle{plain}
\bibliography{mybib}
\end{document}